\documentclass[11pt]{article}

\usepackage{amsfonts}
\usepackage{amscd}
\usepackage{amssymb}
\usepackage{amsthm}
\usepackage{amsmath}
\usepackage{stmaryrd}
\usepackage{graphicx}
\usepackage{color}
\usepackage{verbatim}
\usepackage{mathrsfs}
\usepackage{dsfont}
\usepackage{tikz}
\usepackage{subfigure}

\theoremstyle{plain}
\newtheorem{thm}{Theorem}[section]
\newtheorem{lemma}[thm]{Lemma}
\newtheorem{prop}[thm]{Proposition}
\newtheorem{cor}[thm]{Corollary}

\theoremstyle{definition}
\newtheorem{example}[thm]{Example}
\newtheorem{defn}[thm]{Definition}
\newtheorem{remark}[thm]{Remark}
\theoremstyle{remark}
\newtheorem*{ack}{Acknowledgment}
\numberwithin{equation}{section}

\def\cC{\mathcal{C}}

\def\FF{\mathbb{F}}

\def\RR{\mathbb{R}}

\def\ZZ{\mathbb{Z}}

\def\fc{\mathfrak{c}}

\def\fo{\mathfrak{o}}

\def\fr{\mathfrak{r}}
\def\fs{\mathfrak{s}}
\def\ft{\mathfrak{t}}

\def\bt{\mathbf{t}}

\def\Aut{\mathrm{Aut}}

\newcommand{\vect}[1]{\boldsymbol{#1}}

\def\la{\lambda}

\colorlet{lightgray}{black!20}

\makeatletter
\renewcommand{\@makefnmark}{\mbox{\textsuperscript{}}}
\makeatother

\title{Regular sequences and random walks in affine buildings}
\author{
J. Parkinson \and W. Woess
}

\date{\today}

\begin{document}

\maketitle

\begin{abstract}
We define and characterise regular sequences in affine buildings, thereby giving the $p$-adic analogue of the fundamental work of Kaimanovich~on regular sequences in symmetric spaces. As applications we prove limit theorems for random walks on affine buildings and their automorphism groups.
\let\thefootnote\relax
\footnote{Supported by Austrian Science Fund projects FWF-P19115, FWF-P24028 and FWF-W1230}
\footnote{and by Australian Research Council grant DP110103205.
}
\end{abstract}

\section*{Introduction}

The celebrated \textit{Multiplicative Ergodic Theorem} of Oseledets \cite{O} establishes 
conditions for Lyapunov regularity of random real matrices. In particular, under 
a finite first moment assumption the product of independent identically distributed real 
random matrices behaves ``asymptotically'' like the sequence of powers of some fixed positive 
definite symmetric matrix~$\Lambda$, and the Lyapunov exponents are the logarithms of the 
eigenvalues of~$\Lambda$.

Let $S=GL_r(\RR)/O_r(\RR)$ be the symmetric space associated with $GL_r(\mathbb{R})$, and 
let $g_1,g_2,\ldots$ be a stationary sequence in $GL_r(\RR)$ with finite first moment. 
Kaimanovich \cite{kaimanovich} observed that Lyapunov regularity of the sequence 
$(g_1\cdots g_n)_{n\geq 1}$ of products is equivalent to the existence of a unit speed 
geodesic $\gamma:[0,\infty)\to S$ in $S$ and a number $a\geq 0$ such that 
\begin{align*}
\lim_{n\to\infty}\frac{d\bigl(x_n,\gamma(an)\bigr)}{n}=0,\quad
\textrm{where $\; x_n=g_1\cdots g_no\;$ with  $\;o=O_r(\mathbb{R})\,$, 
the basepoint of $\;S$}.
\end{align*}
This lead naturally to the definition of \textit{regularity} of a sequence in the symmetric 
space of an arbitrary Lie type group: A sequence $(x_n)_{n\geq 0}$ 
in $S$ is \textit{regular} if there is a unit speed geodesic $\gamma:[0,\infty)\to S$ and 
a number $a\geq 0$ such that $d\bigl(x_n,\gamma(an)\bigr)=o(n)$. 
In \cite[Theorems~2.1 and~2.4]{kaimanovich}, Kaimanovich obtains a complete characterisation 
of regular sequences in terms of spherical coordinates and horospheric coordinates 
in the symmetric space. As a consequence, Kaimanovich obtains a Multiplicative Ergodic 
Theorem for noncompact semisimple real Lie groups with finite centre, generalising 
Oseledets' Theorem to arbitrary Lie type.

A striking feature of Kaimanovich's analysis is that it is entirely geometric in nature, 
converting the statement about Lyapunov regularity of matrices to the geometric statement 
about being sublinearly close to a geodesic in the associated symmetric space. 
These ideas have since been extended much further. In \cite[Theorem~2.1]{KM}, Karlsson and 
Margulis prove a Multiplicative Ergodic Theorem that applies to random walks on the isometry 
group of a uniformly convex, Busemann nonpositively curved, complete metric space 
(for example, a $\mathrm{CAT}(0)$ space). In \cite{KL}, Karlsson and Ledrappier give 
a general version of the Multiplicative Ergodic Theorem in terms of Busemann functions.

Our main aim here is to carry Kaimanovich's characterisation of regular sequences 
across to the $p$-adic case of semisimple Lie groups over non-archimedean local fields. 
In this setting the symmetric space is replaced by the \textit{affine Bruhat-Tits building} 
of the group. In fact we will work more generally with arbitrary affine buildings~$\Delta$ 
(including those that do not arise from group constructions, c.f. Ronan \cite{ronanconstruction}). 
Affine buildings are $\mathrm{CAT}(0)$ spaces, and by analogy with the symmetric space case 
we define a sequence $(x_n)_{n\geq 0}$ in $\Delta$ to be \textit{regular} if there exist a 
unit speed geodesic $\gamma:[0,\infty)\to \Delta$ and a number $a\geq 0$ such that
$$
d\bigl(x_n,\gamma(an)\bigr)=o(n).
$$
In Theorem~\ref{thm:main} we give a characterisation of regular sequences in affine buildings 
in the spirit of Kaimanovich's original symmetric space characterisation. 
This characterisation is in terms of a natural \textit{vector distance} in the building 
(the analogue of spherical coordinates in the symmetric space), and in terms of a 
\textit{vector Busemann function} (the analogue of horospheric coordinates in the 
symmetric space). These results can also be seen as extensions into higher rank of 
results of Cartwright, Kaimanovich and Woess \cite{CKW}, where regular sequences in 
trees are studied (these are the simplest affine buildings).

As applications of our characterisation we prove limit theorems and convergence theorems 
for the right random walk on the automorphism group of an affine building. In particular, 
these results apply to `groups of $p$-adic type' (see Macdonald \cite{macsph}). We also give 
limit theorems and convergence theorems for \textit{semi-isotropic} random walks 
on the buildings themselves. In this case there is not necessarily any underlying group 
structure. For example, by a free construction of Ronan \cite{ronanconstruction}, there 
are $\tilde{A}_2$ buildings with trivial automorphism group. 
However we note that by a fundamental theorem of Tits \cite{T} (see also Weiss \cite{weiss}), 
all irreducible affine buildings of rank $4$ or more are ``classical'' and arise from a 
group construction. 

Regarding random walks on $p$-adic groups and affine buildings we mention the works 
of Sawyer \cite{ss}, Cartwright, Kaimanovich and Woess \cite{CKW}, and Brofferio \cite{SB} concerning trees,
and in higher dimension the works of Tolli \cite{tolli}, Lindlbauer and Voit \cite{LV}, 
Cartwright and Woess \cite{CW}, Parkinson~\cite{P3}, Shapira~\cite{BS}, and Parkinson and Shapira~\cite{PS}.
Further references can be found in those papers.

This paper is organised as follows. In Section~\ref{sect:background} we give the relevant background on Coxeter groups and the Coxeter complex. In Section~\ref{sect:buildings} we give background on affine buildings, and we define vector distances and vector Busemann functions in the affine building. In Section~\ref{sect:regularity} we define regular sequences in affine buildings, and present our main theorem characterising regular sequences in the spirit of Kaimanovich's original characterisation for symmetric spaces. In Section~\ref{sect:applications} we give applications of our main result to random walks on $p$-adic Lie groups and affine buildings. 

\begin{ack} We would like to warmly thank Vadim Kaimanovich for helpful discussions 
and suggestions on this paper. 
\end{ack}

\section{Coxeter groups and the Coxeter complex}\label{sect:background}

\textit{Coxeter groups} form the backbone of the more sophisticated \textit{buildings} 
which are the subject of this paper. In this section we give some relevant background
 on Coxeter groups, focussing on the special case of \textit{affine Weyl groups}.

\subsection{Affine Coxeter groups}\label{sect:roots}

A \textit{Coxeter system} $(W,S)$ is a group $W$ generated by a set $S$ with relations
$$
s^2=1\qquad\textrm{and}\qquad (st)^{m_{st}}=1\qquad\textrm{for all $\;s,t\in S\;$ with $\;s\neq t\,$,}
$$ 
where $m_{st}=m_{ts}\in\ZZ_{\geq 2}\cup\{\infty\}$ for all $s\neq t$. As is standard, 
we often say that $W$ is a \textit{Coxeter group} when the generating set $S$ is implied. 
The \textit{length} of $w\in W$ is
$$
\ell(w)=\min\{n\geq 0\mid w=s_1\cdots s_n\textrm{ with }s_1,\ldots ,s_n\in S\}.
$$
An expression $w=s_1\cdots s_n$ with $n=\ell(w)$ is called a \textit{reduced expression} 
for~$w$. A Coxeter system $(W,S)$ is \textit{irreducible} if there is no partition 
$S=S_1\cup S_2$ into nonempty disjoint sets $S_1$ and $S_2$ with $ss'=s's$ 
(that is, $m_{ss'}=2$) for all $s\in S_1$ and $s'\in S_2$. 

A Coxeter group~$W$ is \textit{affine} if it is not finite, but contains a normal Abelian 
subgroup such that the corresponding quotient group is a finite group. All irreducible 
affine Coxeter groups can be constructed as an \textit{affine Weyl group} associated to 
a \textit{root system}. This construction realises the affine Coxeter group as a group 
of reflections in (affine) hyperplanes in an Euclidean space. We outline this explicit 
construction below (the canonical reference is Bourbaki~\cite{bourbaki}). 

Let $E$ be an $r$-dimensional real vector space with inner product $\langle\cdot,\cdot\rangle$. 
If $\alpha\in E\backslash\{0\}$ and $k\in\RR$ let
$
H_{\alpha,k}=\{x\in E\mid \langle x,\alpha\rangle=k\}.
$
This affine hyperplane is parallel to the linear hyperplane 
$H_{\alpha}=H_{\alpha,0}=(\mathbb{R}\alpha)^{\perp}$. The orthogonal reflection 
in the hyperplane $H_{\alpha,k}$ is given by the formula
$$
s_{\alpha,k}(x)=x-(\langle x,\alpha\rangle-k)\alpha^{\vee},\qquad
\textrm{where \quad$\alpha^{\vee}=2\alpha/\langle\alpha,\alpha\rangle$}.
$$
Write $s_{\alpha}=s_{\alpha,0}$, and for $\lambda\in E$ let 
$t_{\lambda}\in\mathrm{Aff}(E)$ be the translation $t_{\lambda}(x)=x+\lambda$.

Let $R$ be a \textit{root system} in~$E$ (see \cite{bourbaki}). Thus $R$ is a finite 
set of non-zero vectors such that (i) $R$ spans $E$, (ii) if $\alpha\in R$ and 
$k\alpha\in R$ then $k=\pm 1$, (iii) if $\alpha,\beta\in R$ then $s_{\alpha}(\beta)\in R$, 
and (iv) if $\alpha,\beta\in R$ then $\langle \alpha,\beta^{\vee}\rangle\in \mathbb{Z}$. 
Assume further that $R$ is \textit{irreducible}, and so there is no partition 
$R=R_1\cup R_2$ with $R_1$ and $R_2$ nonempty such that $\langle\alpha,\beta\rangle=0$ 
for all $\alpha\in R_1$ and all $\beta\in R_2$. Let $\alpha_1,\ldots,\alpha_r\in R$ be 
a fixed choice of \textit{simple roots}. Thus every $\alpha\in R$ can be written as a 
linear combination of $\alpha_1,\ldots,\alpha_r$ with integer coefficients which are 
either all nonpositive, or all nonnegative. Those roots whose coefficients are all 
nonnegative are called \textit{positive roots}, and the set of all positive roots 
is denoted~$R^+$. Then $R=R^+\cup(-R^+)$.

The \textit{Weyl group} of $R$ is the subgroup $W_0$ of $\mathrm{GL}(E)$ generated 
by the reflections $s_{\alpha}$ with $\alpha\in R$. Let $s_j=s_{\alpha_j}$ for $j=1,\ldots,r$. 
The group $W_0$ is generated by the reflections $s_1,\ldots,s_r$, and the order of 
the product $s_is_j$ is $m_{ij}$, where the angle between $\alpha_i$ and $\alpha_j$ 
is $\pi-\pi/m_{ij}$. Thus $W_0$ is a finite Coxeter group relative to the generators 
$s_1,\ldots,s_r$. There is a unique element $w_0\in W_0$ of maximal length 
(the \textit{longest element} of~$W_0$).

The \textit{affine Weyl group} of $R$ is the subgroup $W$ of $\mathrm{Aff}(E)$ generated 
by the reflections $s_{\alpha,k}$ with $\alpha\in R$ and $k\in\ZZ$. 
Let $\varphi\in R$ be the \textit{highest root} of $R$ (the \textit{height} of the 
root $\alpha=a_1\alpha_1+\cdots+a_r\alpha_r$ is $\mathrm{ht}(\alpha)=a_1+\cdots+a_r$, 
and $\varphi$ is the unique root with greatest height). Let $s_0=s_{\varphi,1}$. 
The group $W$ is generated by the reflections $s_0,s_1,\ldots,s_r$, and the order 
of $s_0s_j$ is $m_{0j}$, where the angle between $\varphi$ and $\alpha_j$ is $\pi-\pi/m_{0j}$. 
Thus $W$ is an infinite Coxeter group relative to the generators $s_0,s_1,\ldots,s_r$. 
Since $s_{\alpha,k}=t_{k\alpha^{\!\vee}}\,s_{\alpha}$ we have
$$
W=Q\rtimes W_0,\qquad\textrm{where}\qquad Q=\ZZ\alpha_1^{\vee}+\cdots+\ZZ\alpha_r^{\vee}\,.
$$
Thus $W$ is an affine Coxeter group. All irreducible affine Coxeter groups arise in this 
way for some choice of root system. The lattice $Q$ is the \textit{coroot lattice} of $R$. 
Let $\omega_1,\ldots,\omega_r$ be the basis of $E$ dual to $\alpha_1,\ldots,\alpha_r$ 
(so $\langle\omega_i,\alpha_j\rangle=\delta_{ij}$). The \textit{coweight lattice} of $R$ is 
$$
P=\ZZ\omega_1+\cdots+\ZZ\omega_r=
\{\lambda\in E\mid \langle\lambda,\alpha\rangle\in\ZZ \;\textrm{for all $\alpha\in R$}\},
$$
and elements of $P$ are called \textit{coweights} of $R$ (with $\omega_1,\ldots,\omega_r$ 
being \textit{fundamental coweights}). We have $Q\subseteq P$, and for any lattice $L$ 
with $Q\subseteq L\subseteq P$ we define a group $W_L=L\rtimes W_0.$
In particular $W_Q=W$, however in general $W_L$ is not a Coxeter group. 
We have $W_L=W\rtimes (L/Q)$ with $L/Q$ finite and Abelian. We call $W_L$ an 
\textit{extended affine Weyl group}.

\subsection{The Coxeter complex}\label{sect:coxcomp}

Let $R$ be an irreducible root system with Weyl group~$W_0$, affine Weyl group~$W$, coroot 
lattice~$Q$, and coweight lattice~$P$ (as in the previous section). Let $\mathscr{H}$ be 
the family of hyperplanes $H_{\alpha,k}$ with $\alpha\in R^+$ and $k\in\mathbb{Z}$. 
The closures of the open connected components of $E\backslash \mathscr{H}$ are geometric 
simplices of dimension~$r$. The extreme points of these geometric simplices are 
\textit{vertices}, and the resulting simplicial complex $\Sigma$ is the 
\textit{Coxeter complex of~$W$}. We will rather freely interchange between regarding 
$\Sigma$ as an ``abstract'' simplicial complex (cf. Abramenko and Brown \cite{AB}) and a 
``geometric'' simplicial complex (embedded in the Euclidean space~$E$). The extended affine 
Weyl groups $W_L$ act on  $\Sigma$ by simplicial complex automorphisms. Thus, in particular, 
each $\lambda\in P$ is a vertex of~$\Sigma$, however in general $P$ is a strict subset of the vertex set. The vertices $\lambda\in P$ are called the \textit{special vertices} 
of~$\Sigma$. 

The maximal dimensional simplices of~$\Sigma$ are called \textit{chambers} (thus, each 
chamber has exactly $r+1$ vertices). The \textit{fundamental chamber} is the geometric simplex
$$
\fc_0=\{x\in E\mid \langle x,\alpha_i\rangle\geq 0\;\textrm{for all}\; i=1, \dots, r \;
\textrm{and}\;\langle x,\varphi\rangle\leq 1\},
$$
where $\varphi$ is the highest root of~$R$. This chamber is bounded by the hyperplanes 
$H_{\alpha_i}$ ($1\leq i\leq r$) and $H_{\varphi,1}$. The affine Weyl group $W$ acts simply 
transitively on the set of chambers of~$\Sigma$, and we usually identify the chambers 
with~$W$ (by $w\fc_0\leftrightarrow w$).

The vertex set of $\fc_0$ is
$
\{0\}\cup\{\omega_i/m_i\mid 1\leq i\leq r\},
$
where $\varphi=m_1\alpha_1+\cdots+m_r\alpha_r$ is the highest root. Define the \textit{type} 
of the vertex $0$ to be $\tau(0)=0$ and the type of the vertex $\omega_i/m_i$ to be 
$\tau(\omega_i/m_i)=i$ for $1\leq i\leq r$. This extends uniquely to a type function 
$\tau:\Sigma\to 2^{\{0,1,\ldots,r\}}$ making $\Sigma$ into a \textit{labelled simplicial complex} 
(where each chamber has exactly one vertex of each type). The type of a simplex $\sigma\in\Sigma$ 
is $\tau(\sigma)=\{\tau(x)\mid x\in \sigma\}$. The action of $W$ on $\Sigma$ preserves types, 
and $Q$ is the set of all type~$0$ vertices of~$\Sigma$.

The hyperplanes in $\mathscr{H}$ are called the \textit{walls} of $\Sigma$. Each wall 
$H_{\alpha,k}$ of $\Sigma$ determines two \textit{half apartments}, 
$H_{\alpha,k}^+=\{x\in E\mid \langle x,\alpha\rangle\geq k\}$ and 
$H_{\alpha,k}^-=\{x\in E\mid \langle x,\alpha\rangle\leq k\}$ (the term ``apartment'' comes 
from the building language of the next section). 

The \textit{fundamental sector} of $\Sigma$ is the geometric cone
$$
\fs_0=\bigcap_{i=1}^rH_{\alpha_i,0}^+=\{x\in E\mid \langle x,\alpha_i\rangle\geq 0\;
\textrm{for all}\; i=1,\dots, r\}.
$$
Points $x\in\fs_0$ are called \textit{dominant}, and it is convenient at times to 
write $E^+=\fs_0$. Any set of the form $w\fs_0$ with $w\in W_P$ is called a \textit{sector} 
of $\Sigma$. The \textit{base vertex} of $\fs_0$ is $0$, and the \textit{sector panels} 
of $\fs_0$ are the half lines $H_{\alpha_i}\cap \fs_0=\RR_{\geq 0}\,\omega_i\,$, 
$1\leq i\leq r$. The base vertex and sector panels of $w\fs_0$ are the translates of 
those for~$\fs_0$, and so in particular sectors are always based at special vertices. 
Two sectors are \textit{adjacent} if they have the same base vertex and share a sector panel. 
Since the group $W_0$ acts simply transitively on the set of sectors with base vertex~$0$, for 
each $\mu\in E$ there is a unique dominant element $\mu^+\in E^+$ in the orbit $W_0\mu$. 

\newpage

\begin{example} Figure~\ref{fig:C2} shows the $\tilde{C}_2$ case. 
\begin{figure}[!h]
\begin{minipage}[t]{0.45\linewidth}
\vspace{0pt}
\begin{center}
\begin{tikzpicture}[scale=1.35]
\path [fill=lightgray] (0,0) -- (2.25,0) -- (2.25,2.25) -- (0,0);
\draw (-2.25,-2) -- (2.25,-2); 
\draw (-2.25,-1) -- (2.25,-1);
\draw (-2.25,0) -- (2.25,0);
\draw (-2.25,1) -- (2.25,1);
\draw (-2.25,2) -- (2.25,2);
\draw (-2,-2.25) -- (-2,2.25);
\draw (-1,-2.25) -- (-1,2.25);
\draw (0,-2.25) -- (0,2.25);
\draw (1,-2.25) -- (1,2.25);
\draw (2,-2.25) -- (2,2.25);
\draw (-2.25,1.75) -- (-1.75,2.25);
\draw (-2.25,-0.25) -- (0.25,2.25);
\draw (-2.25,-2.25) -- (2.25,2.25);
\draw (-0.25,-2.25) -- (2.25,0.25);
\draw (1.75,-2.25) -- (2.25,-1.75);
\draw (-2.25,-1.75) -- (-1.75,-2.25);
\draw (-2.25,0.25) -- (0.25,-2.25);
\draw (-2.25,2.25) -- (2.25,-2.25);
\draw (-0.25,2.25) -- (2.25,-0.25);
\draw (1.75,2.25) -- (2.25,1.75);
\node at (0,0) {\Large{$\bullet$}};
\node at (-2,0) {\Large{$\bullet$}};
\node at (2,0) {\Large{$\bullet$}};
\node at (-2,-2) {\Large{$\bullet$}};
\node at (0,-2) {\Large{$\bullet$}};
\node at (2,-2) {\Large{$\bullet$}};
\node at (0,2) {\Large{$\bullet$}};
\node at (-2,2) {\Large{$\bullet$}};
\node at (2,2) {\Large{$\bullet$}};
\node at (-1,1) {\large{$\Box$}};
\node at (1,1) {\large{$\Box$}};
\node at (-1,-1) {\large{$\Box$}};
\node at (1,-1) {\large{$\Box$}};
\node at (1.5,-1.8) {\small{$\alpha_1^{\vee}$}};
\node at (0.5,1.8) {\small{$\alpha_2^{\vee}$}};
\node at (0.65,0.25) {\small{$\mathfrak{c}_0$}};
\node at (1.6,0.15) {\small{$\omega_1$}};
\node at (1.4,1.15) {\small{$\omega_2$}};
\end{tikzpicture}
\end{center}
\end{minipage}\hfill
\begin{minipage}[t]{0.53\linewidth}
\vspace{0pt}
Let $R=\pm\{e_1-e_2,e_1+e_2,2e_1,2e_2\}$ be a root system of type $C_2$. 
The roots $\alpha_1=e_1-e_2$ and $\alpha_2=2e_2$ are simple roots. 
Thus $\alpha_1^{\vee}=e_1-e_2$ and $\alpha_2^{\vee}=e_2$ (as shown). 
The fundamental coweights are $\omega_1=e_1$ and $\omega_2=\frac{1}{2}(e_1+e_2)$. 
The highest root is $2\alpha_1+\alpha_2$, and the vertices of $\fc_0$ are 
$\{0,\omega_1/2,\omega_2\}$. The type $0$ vertices (respectively type $2$ vertices) 
are marked with $\bullet$ (respectively $\Box$). The remaining vertices have type~$1$, 
and are not special vertices. The fundamental sector $\fs_0$ is shaded. The coroot 
lattice $Q$ is the set of type $0$ vertices and the coweight lattice is the set of 
all special vertices.
\end{minipage}
\caption{The $\tilde{C}_2$ Coxeter complex}\label{fig:C2}
\end{figure}
\end{example}

\section{Affine buildings}\label{sect:buildings}

The theory of \textit{buildings} grew from the fundamental work of Jacques Tits starting 
in the 1950s. The initial impetus was to give a uniform description of semisimple Lie groups 
and algebraic groups by associating a geometry to each such group. This ``geometry of 
parabolic subgroups'' later became known as the \textit{building} of the group \cite{titsbook}. 
Since their invention, buildings and related geometries have enjoyed extensive study and 
development, and have found applications in many areas of mathematics; see Abramenko and Brown 
\cite{AB}, Ronan \cite{ronan}, and the survey by Ji \cite{ji}. Our main reference for this 
section is~\cite{AB}.

We begin with the definition of an affine building, along with some basic definitions from 
the theory of buildings. After this we define vector distances in the affine building, 
and then we introduce and develop the theory of vector valued Busemann functions on the affine 
building.

\subsection{Definitions}\label{sect:defns}

We adopt Tits' definition of buildings as simplicial complexes satisfying certain axioms 
(cf.~\cite{AB}). Recall that a \textit{simplicial complex} with vertex set $V$ is a 
collection $\Delta$ of finite subsets of $V$ (called \textit{simplices}) such that for every 
$v\in V$, the singleton $\{v\}$ is a simplex, and every subset of a simplex $\sigma$ is 
a simplex (a \textit{face} of $\sigma$). If $\sigma$ is a simplex which is not a proper subset 
of any other simplex then $\sigma$ is a \textit{chamber} of~$\Delta$. 

\begin{defn}\label{defn:building} Let $(W,S)$ be an irreducible affine Coxeter system 
with Coxeter complex $\Sigma$. A \textit{building of type $(W,S)$} is a nonempty simplicial 
complex $\Delta$ which is the union of subcomplexes called \textit{apartments}, each isomorphic 
to~$\Sigma$, such that
\begin{enumerate}
\item[(B1)] given any two simplices of $\Delta$ there is an apartment containing both of them, and
\item[(B2)] if apartments $A,A'$ of $\Delta$ contain a common chamber then there is a unique simplicial 
complex isomorphism $\psi:A\to A'$ fixing every vertex of $A\cap A'$.
\end{enumerate}
\end{defn}

Fix, once and for all, an apartment of $\Delta$ and identify it with $\Sigma$. Thus we 
regard $\Sigma$ as an apartment of $\Delta$, the ``standard apartment''. Let 
$I=\{0,1,\ldots,r\}$. The type function on $\Sigma$ extends uniquely to a 
type function $\tau:\Delta\to 2^I$ making $\Delta$ into a labelled simplicial complex. 
The isomorphism in the second building axiom is necessarily type preserving.

The cardinality $|\sigma|$ of a simplex $\sigma\in \Delta$ is called the \textit{dimension} 
of the simplex, and the \textit{codimension} of $\sigma$ is $r+1-|\sigma|$. The chambers are 
the maximal dimensional simplices, and a \textit{panel} is a codimension~$1$ simplex. 
Chambers $c$ and $d$ of $\Delta$ are \textit{$i$-adjacent} (written $c\sim_i d$) if they 
share a panel~$\pi$ of type $I\backslash\{i\}$. A \textit{gallery of type $(i_1,\ldots,i_n)$} 
from $c$ to $d$ is a sequence of chambers
$$
c=c_0\sim_{i_1}c_1\sim_{i_2}\cdots\sim_{i_n}c_n=d\quad
\textrm{with $\;c_{k-1}\neq c_k\;$ for $\;1\leq k\leq n\,$}.
$$
This gallery has minimal length amongst all galleries from $c$ to $d$ if and only if 
$s_{i_1}\cdots s_{i_n}$ is a reduced expression in~$W$.

A building is called \textit{thick} if every panel is contained in at least~$3$ chambers. 
For most of this paper we do not explicitly require $\Delta$ to be thick, but we note that 
thick buildings are certainly the most interesting buildings.

As a minor technical point, we will always take the \textit{complete} apartment system of 
$\Delta$. Thus every subcomplex of $\Delta$ which is isomorphic to $\Sigma$ is taken to be 
an apartment of~$\Delta$. See \cite[\S4.5]{AB}. The definitions of walls, sectors, sector 
panels, and half apartments of~$\Sigma$ transfer across to the building $\Delta$. For 
example, a \textit{sector} of $\Delta$ is a subset of $\Delta$ which is a sector in 
some apartment of $\Delta$, and so on.

\begin{example} Figure~\ref{fig:piece} shows a simplified picture of an affine $\tilde{A}_2$ 
building. If the building is thick then the branching must actually occur along \textit{every} 
wall, and so the picture is rather incomplete. The chambers are the triangles, and the 
panels are the edges of the triangles. There are $6$ apartments shown (for example, the 
``horizontal'' sheet), and a sector is shaded. 
\begin{figure}[h]
\begin{center}
\begin{tikzpicture}[scale=0.6]
\path [fill=lightgray] (-6,0) -- (-4,-2) -- (-2,0);
\path [fill=lightgray] (5,-3) -- (-3,-3) -- (-3.4,-2.6) -- (-3,-2.5) -- (-1,-2.5) -- (1,-2) -- (3,-2) -- (5,-1.5) -- (6.5,0) -- (8,0) -- (5,-3);
\draw (-5,-3) -- (-3,-2.5) -- (-1,-2.5) -- (1,-2) -- (3,-2) -- (5,-1.5) -- (10,3.5) -- (8,3) -- (6,3) -- (4,2.5) -- (2,2.5) -- (0,2) -- (-6,2) -- (-11,-3) -- (5,-3) -- (10,2) -- (8.5,2);
\draw (-7,-3) -- (-9,-4) -- (-11,-4.5) -- (-13,-4.5) -- (-8,0.5) -- (-7.5,0.5);
\draw (-9,-4) -- (-8.1,-3.1);
\draw (-5,-3) -- (0,2);
\draw (-3,-2.5) -- (2,2.5);
\draw (-1,-2.5) -- (4,2.5);
\draw (1,-2) -- (6,3);
\draw (-9,-3) -- (-4,2);
\draw (-7,-3) -- (-2,2);
\draw (-10,-2) -- (-4,-2) -- (-2,-1.5) -- (0,-1.5) -- (2,-1) -- (4,-1) -- (6,-0.5);
\draw (-9,-1) -- (-3,-1) -- (-1,-0.5) -- (1,-0.5) -- (3,0) -- (5,0) -- (7,0.5);
\draw (-8,0) -- (-2,0) -- (0,0.5) -- (2,0.5) -- (4,1) -- (6,1) -- (8,1.5);
\draw (-7,1) -- (-1,1) -- (1,1.5) -- (3,1.5) -- (5,2) -- (7,2) -- (9,2.5);
\draw (-10,-2) -- (-9,-3);
\draw (-9,-1) -- (-7,-3);
\draw (-8,0) -- (-5,-3);
\draw (-7,1) -- (-4,-2);
\draw (-6,2) -- (-3,-1);
\draw (-4,2) -- (-2,0);
\draw (-2,2) -- (-1,1);
\draw (8,3) -- (9,2.5);
\draw (6,3) -- (7,2) -- (8,1.5);
\draw (4,2.5) -- (5,2) -- (6,1) -- (7,0.5);
\draw (2,2.5) -- (3,1.5) -- (4,1) -- (5,0) -- (6,-0.5);
\draw (0,2) -- (1,1.5) -- (2,0.5) -- (3,0) -- (4,-1) -- (5,-1.5);
\draw (3,-2) -- (8,3);
\draw (-1,1) -- (0,0.5) -- (1,-0.5) -- (2,-1) -- (3,-2);
\draw (-2,0) -- (-1,-0.5) -- (0,-1.5) -- (1,-2);
\draw (-3,-1) -- (-2,-1.5) -- (-1,-2.5);
\draw (-4,-2) -- (-3,-2.5);
\draw (7.7,1) -- (9,1);
\draw (6.7,0) -- (8,0);
\draw (5.7,-1) -- (7,-1);
\draw (3.3,-2) -- (6,-2);
\draw (3,-3) -- (4.2,-1.8);
\draw (1,-3) -- (1.9,-2.1);
\draw (-1,-3) -- (-0.5,-2.5);
\draw (-3,-3) -- (-2.6,-2.6);
\draw (-3,-3) -- (-3.3,-2.7);
\draw (-1,-3) -- (-1.4,-2.6);
\draw (1,-3) -- (0.3,-2.3);
\draw (3,-3) -- (2.1,-2.1);
\draw (5,-3) -- (3.9,-1.9);
\draw (6,-2) -- (5.4,-1.4);
\draw (7,-1) -- (6.4,-0.4);
\draw (8,0) -- (7.4,0.6);
\draw (9,1) -- (8.4,1.6);
\draw (-9,-4) -- (-10,-3.5) -- (-8.3,-3.1);
\draw (-11,-4.5) -- (-9.7,-3.2);
\draw (-11,-4.5) -- (-12,-3.5) -- (-10,-3.5) -- (-10.4,-3.1);
\draw (-10.83,-2.67) -- (-11,-2.5);
\draw (-9.83,-1.67) -- (-10,-1.5);
\draw (-8.83,-0.67) -- (-9,-0.5);
\draw (-7.83,0.33) -- (-8,0.5);
\draw (-11,-2.5) -- (-10.65,-2.5);
\draw (-10,-1.5) -- (-9.65,-1.5);
\draw (-9,-0.5) -- (-8.65,-0.5);
\end{tikzpicture}
\end{center}
\caption{A piece of an $\tilde{A}_2$ building}
\label{fig:piece}
\end{figure}
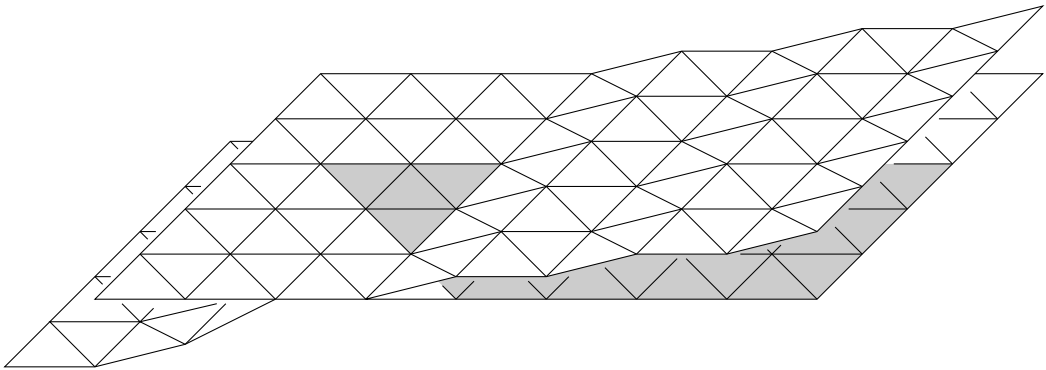
\end{example}

The affine building $\Delta$ can be viewed as a metric space $(|\Delta|,d)$ in a standard 
way (the \textit{geometric realisation}, see \cite[\S11.2]{AB}). Each apartment is 
isomorphic to an affine Coxeter complex and hence can be viewed as a metric space by 
the construction in Section~\ref{sect:roots}. Then the building axioms ensure that these 
metrics can be `glued together' in a unique way to give a metric on~$|\Delta|$. All 
isomorphisms in the building axioms can be taken to be isometries. Since we will henceforth 
regard $\Delta$ as a metric space, we will drop the vertical bar notation, and simply 
write $\Delta$ for the metric space $(|\Delta|,d)$.

By \cite[Theorem~11.16]{AB} the building $\Delta$ is a $\mathrm{CAT}(0)$ space. Thus, $\Delta$ 
is a complete metric space with unique geodesics, satisfying the 
\textit{negative curvature inequality}: If $x,y\in \Delta$, $t\in [0,1]$, and $p(t)$ is the 
unique point on the geodesic $[x,y]$ with $d\bigl(x,p(t)\bigr)=t\,d(x,y)$ then
\begin{align}\label{eq:negcurv}
d^2\bigl(z,p(t)\bigr)\leq(1-t)\,d^2(z,x)+t\, d^2(z,y)-t(1-t)\,d^2(x,y)\quad
\textrm{for all} \;z\in \Delta\,.
\end{align}
 
Since $\Delta$ is $\mathrm{CAT}(0)$ we define the \textit{visibility boundary} $\partial \Delta$ 
in the usual way as the set of equivalence classes of rays (with two rays being 
\textit{parallel} if the distance between them is bounded). The standard topology makes 
$\overline{\Delta}=\Delta\cup\partial \Delta$ into a compact Hausdorff space 
(see Bridson and Haefliger \cite[\S II.8.5]{bridson}). Points of the visibility boundary are called 
\textit{ideal points} of $\Delta$. Given $\xi\in\partial \Delta$ and $x\in \Delta$, 
there is a unique ray in the class $\xi$ with base point $x$ (\cite[Proposition~II.8.2]{bridson} 
or \cite[Lemma~11.72]{AB}). We sometimes denote this ray by $[x,\xi)$. Thus one may think 
of $\partial \Delta$ as ``all rays based at $x$'' for any fixed~$x\in\Delta$.

Two fundamental facts from \cite[Chapter~11]{AB} are: 
\begin{enumerate}
\item[(S1)] Given a sector $\fs$ and a chamber $c$ of $\Delta$, there is a subsector $\fs'$ of 
$\fs$ such that $\fs'\cup c$ lies in an apartment, and 
\item[(S2)] given sectors $\fs$ and $\ft$ of $\Delta$, there are subsectors 
$\fs'\subseteq \fs$ and $\ft'\subseteq \ft$ such that $\fs'\cup\ft'$ lies in 
an apartment. 
\end{enumerate}

The following configurations of half apartments, sectors, and chambers will arise later in this paper. 

\begin{lemma}\label{lem:config} Let $c$ be a chamber of~$\Delta$, let $\fs$ be a 
sector of~$\Delta$, and let $H^+$ be a half apartment of $\Delta$ bounded by a 
wall~$H$ of~$\Delta$.
\begin{enumerate}
\item If $c$ has a panel contained in $H$ then there is an apartment containing $H^+\cup c$.
\item If $c$ contains the base vertex of $\fs$ then there is an apartment containing $\fs\cup c$.
\item If the half apartment $H^+$ and the sector $\fs$ intersect exactly along a sector 
panel of $\fs$ then there is an apartment containing $H^+\cup\fs$. 
\end{enumerate}
\end{lemma}

\begin{proof}
The first two statements are routine applications of \cite[Theorem~5.73]{AB} 
(or \cite[Theorem~3.6]{ronan}). For the third statement, using ideas from \cite[\S11.5]{AB} 
one sees that the configuration $H^+\cup\fs$ is isometric to a subset of $W$ (the key point 
here is that for any chambers $c,d\in H^+\cup\fs$ there is a minimal gallery from $c$ to $d$ 
which is contained in $H^+\cup\fs$). Hence $H^+\cup\fs$ is contained in an apartment 
by \cite[Theorem~5.73]{AB}.
\end{proof}

\subsection{Vector distance}

Recall that $E^+=\{x\in E\mid\langle x,\alpha_i\rangle\geq 0\textrm{ for all }1\leq i\leq r\}$. 

\begin{defn} Let $x$ and $y$ be any points of  $\Delta$. The \textit{vector distance} 
$\vect{d}(x,y)\in E^+$ from $x$ to $y$ is defined as follows. By (B1), choose an apartment $A$ containing 
$x$ and $y$ and choose a type preserving isomorphism $\psi:A\to \Sigma$. Then we define
$$
\vect{d}(x,y)=\bigl(\psi(y)-\psi(x)\bigr)^+,
$$
where for $\mu \in E$, we denote by $\mu^+$ the unique element in $W_0\mu\cap E^+$.
\end{defn}

By an argument similar to \cite[Proposition~5.6]{P1}, the building axiom~(B2) implies that 
the value of $\vect{d}(x,y)$ does not depend on the choices of $A$ and $\psi$ in the definition. 
More intuitively, to compute $\vect{d}(x,y)$ one looks at the vector from $x$ to $y$ (in any 
apartment containing $x$ and $y$) and takes the dominant representative of this vector under the 
$W_0$-action. We have 
$
d(x,y)=\|\vect{d}(x,y)\|,
$
where $d(x,y)$ is the metric distance in~$\Delta$.

\subsection{Vector Busemann functions}

\textit{Busemann functions} (see \cite[Definition~8.17]{bridson}) play an important role 
in the theory of $\mathrm{CAT}(0)$ spaces. In this section we will use retractions in the 
building to define vector analogues of Busemann functions for affine buildings. Essentially 
the ``geodesic ray'' in the usual definition of a Busemann function is replaced by a ``sector'' 
in the affine building (see Proposition~\ref{prop:ft}).

\begin{defn} In affine buildings there are two kinds of retractions -- \textit{chamber based}, 
and \textit{sector based}. Let $A$ be an apartment of $\Delta$. Let $c$ be a chamber of $A$ and 
let $\fs$ be a sector of~$A$. The \textit{retractions} $\rho_{A,c}:\Delta\to A$ and 
$\rho_{A,\fs}:\Delta\to A$ are defined as follows. Let $x\in \Delta$.
\begin{enumerate}
\item[(1)] To compute $\rho_{A,c}(x)$ choose an apartment $A'$ containing $c$ and $x$ (using (B1)) and let 
$\psi:A'\to A$ be the isomorphism from (B2) fixing $A'\cap A$. Then 
$
\rho_{A,c}(x)=\psi(x).
$
\item[(2)] To compute $\rho_{A,\fs}(x)$ choose an apartment $A'$ which contains a subsector 
$\fs'$ of $\fs$ and the point~$x$ (using (S1)). Let $\psi:A'\to A$ be the isomorphism from (B2) fixing $A'\cap A$, and define
$
\rho_{A,\fs}(x)=\psi(x).
$
\end{enumerate}
It is readily seen from the building axioms that the values of $\rho_{A,c}(x)$ and $\rho_{A,\fs}(x)$ do not depend on the 
particular choice of the apartment~$A'$ (see~\cite{AB}).
\end{defn}

For chambers $c$ ``sufficiently deep'' in the sector $\fs$ we have
$
\rho_{A,\fs}(x)=\rho_{A,c}(x),
$
although it is not possible to choose $c$ independently of~$x$. Conceptually, $\rho_{A,c}$ 
``radially flattens'' the building onto $A$ from the centre $c$,
and $\rho_{A,\fs}$ flattens the building onto $A$ from a remote centre ``deep in the sector~$\fs$''. 
Note that if $\fs'$ is a sector of $A$ such that $\fs\cap \fs'$ contains a sector, 
then $\rho_{A,\fs'}=\rho_{A,\fs}$.

If $\fs$ is a sector in an apartment $A$, then there is a unique isomorphism 
$\psi_{A,\fs}:A\to \Sigma$ mapping $\fs$ to $\fs_0$ and preserving vector distances (cf. \cite[Lemma~3.2]{sph}). 
That is, $\psi_{A,\fs}(\fs)=\fs_0$ and 
$\vect{d}\bigl(\psi_{A,\fs}(x),\psi_{A,\fs}(y)\bigr)=\vect{d}(x,y)$ 
for all $x,y\in A$. We often use $\psi_{A,\fs}$ to fix a Euclidean coordinate system on an apartment~$A$.

\begin{defn}\label{defn:horof} The \textit{vector Busemann function} associated to the 
sector $\fs$ is the function
$$
\vect{h}_{\fs}:\Delta\to E\quad\textrm{given by}\quad 
\vect{h}_{\fs}(x)=\psi_{A,\fs}\bigl(\rho_{A,\fs}(x)\bigr)
$$
where $A$ is any apartment containing~$\fs$. It follows from the building axioms that 
the value of $\vect{h}_{\fs}$ does not depend on the choice of the apartment $A$ containing~$\fs$. 
We write $\vect{h}=\vect{h}_{\fs_0}$. Thus
$$
\vect{h}:\Delta\to E\qquad\textrm{is given by}\qquad \vect{h}(x)=\rho_{\Sigma,\fs_0}(x),
$$
since $\psi_{\Sigma,\fs_0}:\Sigma\to\Sigma$ is the identity.
\end{defn}

\begin{example}
Let $A$ be the ``horizontal apartment'' in Figure~\ref{fig:piece}, and let $\fs$ be the shaded sector. 
There are two ways to identify $A$ with $\Sigma$ such that $\fs$ is identified with $\fs_0$, 
but only one of these ways (namely $\psi_{A,\fs}$) will preserve vector distance (in this case 
this amounts to looking at~$A$ from ``above'' or ``below''). To compute the vector Busemann function 
$\vect{h}_{\fs}(x)$ one can imagine a strong wind blowing from deep within the sector~$\fs$, 
causing the building to flatten down to the apartment~$A$. Then $\vect{h}_{\fs}(x)$ is the place 
in~$A$ where $x$ flattens to, where $A$ is identified with $\Sigma$ via~$\psi_{A,\fs}$.
\end{example}

Comparing the following proposition with \cite[Definition~8.17]{bridson} shows that our vector 
Busemann function is an analogue of the usual Busemann functions for $\mathrm{CAT}(0)$ spaces 
with ``geodesic rays'' replaced by ``sectors'' (and also with an extra superficial minus sign).

\begin{prop}\label{prop:ft} Let $\fs$ be a sector of $\Delta$ with base point $\fs(0)$. 
For each $\lambda\in E^+$ let $\fs(\lambda)$ be the unique point of $\fs$ with 
$\vect{d}\bigl(\fs(0),\fs(\lambda)\bigr)=\lambda$. Then
$$
\vect{h}_{\fs}(x)=\lim_{\lambda\to\infty}\Bigl(\lambda-\vect{d}\bigl(x,\fs({\lambda})\bigr)\Bigr)
\quad\textrm{for all}\:x\in \Delta\,,
$$
where the limit is taken with $\langle\lambda,\alpha_i\rangle\to\infty$ for each $i=1,\ldots,r$. 
In particular, the limit exists.
\end{prop}

\begin{proof} Let $A$ be an apartment containing $\fs$, and let $\rho=\rho_{A,\fs}$ and 
$\psi=\psi_{A,\fs}$. Thus, by definition, $\vect{h}_{\fs}(x)=\psi(\rho(x))$.
Choose an apartment $A'$ (by (S1)) containing a subsector $\fs'$ of $\fs$ and the point $x\in \Delta$. 
Note that $A'$ contains all points $\fs(\lambda)$ with each $\langle\lambda,\alpha_i\rangle$ 
sufficiently large, and for such $\lambda$ we have 
$$
\vect{d}\Bigl(\rho(x),\rho\bigl(\fs(\lambda)\bigr)\Bigr)=\vect{d}\bigl(x,\fs(\lambda)\bigr),
$$
because $\rho|_{A'}:A'\to A$ is a type preserving isomorphism (and hence preserves vector distance). 
On the other hand, since $\vect{h}_{\fs}=\psi\circ \rho$ and $\psi$ preserves vector distances, we have
$$
\vect{d}\Bigl(\rho(x),\rho\bigl(\fs(\lambda)\bigr)\Bigr)
=\vect{d}\Bigl(\vect{h}_{\fs}(x),\vect{h}_{\fs}\bigl(\fs(\lambda)\bigr)\Bigr)
=\vect{d}\bigl(\vect{h}_{\fs}(x),\lambda\bigr).
$$
For sufficiently large $\lambda$ we have $\vect{d}\bigl(\vect{h}_{\fs}(x),\lambda\bigr)
=\lambda-\vect{h}_{\fs}(x)$ by the Euclidean geometry of $E$ and the fact that 
$\lambda-\vect{h}_{\fs}(x)$ is dominant for large~$\lambda$. Thus for fixed $x\in \Delta$ we have
$
\vect{h}_{\fs}(x)=\lambda-\vect{d}\bigl(x,\fs(\lambda)\bigr)$ for all $\lambda$ with each 
$\langle\lambda,\alpha_i\rangle$ sufficiently large, hence the result.
\end{proof}

\subsection{The building of a Lie group over a non-archimedean local field}\label{sect:ex}

Affine buildings are intimately related to Lie groups over non-archimedean local fields. Specifically, if $G$ is such a group with maximal compact subgroup~$K$, then $G/K$ is a subset of the vertex set of the so called \textit{affine Bruhat-Tits building}~$\Delta$ of~$G$ (see \cite{BT} and \cite{macsph}).

Let us first give the details for the concrete type~$A$ example. Let $\mathbb{F}$ be a non-archimedean local field with ring of integers $\mathfrak{o}$ and 
uniformiser~$\varpi$. Thus $\mathbb{F}$ is either a finite extension of the $p$-adics 
$\mathbb{Q}_p$ or $\mathbb{F}$ is the field of formal Laurent series $\mathbb{F}_q(\!(t)\!)$ 
with coefficients in a finite field~$\mathbb{F}_q$. If $\mathbb{F}=\mathbb{Q}_p$ then 
$\fo=\mathbb{Z}_p$ and $\varpi=p$, and if $\mathbb{F}=\mathbb{F}_q(\!(t)\!)$ then 
$\fo=\mathbb{F}_q[[t]]$ and $\varpi=t$.

Let $G=PGL_{r+1}(\FF)$ and let $K=PGL_{r+1}(\fo)$. 
Write $P=\ZZ^{r+1}+\ZZ\mathbf{1}$, with $\mathbf{1}=(1,\ldots,1)$, and
$P^+=\{\lambda\in \ZZ^{r+1}\mid \lambda_1\geq \cdots\geq \lambda_{r+1}\}+\ZZ\mathbf{1}$. 
For each $\lambda\in P$ let 
$t_{\lambda}=\mathrm{diag}(\varpi^{-\lambda_1},\ldots,\varpi^{-\lambda_{r+1}})$ 
(note that $t_{\lambda+\ZZ\mathbf{1}}=t_{\lambda}$ in $G$). Let $U$ be the upper 
triangular matrices in $G$ with $1$s on the diagonal. The \textit{Cartan} and \textit{Iwasawa} 
decompositions of $G$ are (respectively): 
\begin{align}\label{eq:CIdecomp}
G=\bigsqcup_{\lambda\in P^+}Kt_{\lambda}K\qquad\textrm{and}\qquad G=\bigsqcup_{\mu\in P}Ut_{\mu}K.
\end{align}
The Bruhat-Tits building of $G$ is a simplicial 
complex $\Delta$ with vertex set~$G/K$. In the $1$-skeleton of~$\Delta$, the vertex $gK$ 
is adjacent to the vertex $hK$ if and only if $g^{-1}hK\subseteq Kt_{\omega_i}K$ for some 
$1\leq i\leq r$, where $\omega_i=(1,\ldots,1,0,\ldots,0)+\ZZ\mathbf{1}$ (with $i$ $1$s). 

Each panel of $\Delta$ lies in exactly $|k|+1$ chambers, where $k=\mathfrak{o}/\varpi\mathfrak{o}$ 
is the (finite) residue field of~$\mathbb{F}$. Thus, for example, if 
$\mathbb{F}=\mathbb{F}_q(\!(t)\!)$ then every panel lies in exactly $q+1$ chambers. The vector distance between vertices $gK$ and $hK$ in $\Delta$ is 
\begin{align}\label{eq:vbuild}
\vect{d}(gK,hK)=\lambda\qquad\textrm{if and only if}\qquad g^{-1}hK\subseteq Kt_{\lambda}K,
\end{align}
and since each $u\in U$ stabilises a subsector of the fundamental sector of $\Delta$, 
the Busemann function $\vect{h}$ is given by 
\begin{align}\label{eq:hbuild}
\vect{h}(gK)=\mu\qquad\textrm{if and only if}\qquad gK\subseteq Ut_{\mu}K.
\end{align}

More generally, let $R$ be a root system with coroot lattice $Q$ and coweight lattice $P$, and let $L$ be a lattice with $Q\subseteq L\subseteq P$. Let $\mathbb{F}$ be a non-archimedean local field as above. Let $G=G(\mathbb{F})$ be the Chevalley group over~$\mathbb{F}$ with root datum $(R,L)$ (see \cite{steinberg} or \cite{carter}). Thus $G$ is generated by elements $x_{\alpha}(f)$ and $h_{\lambda}(g)$ with $\alpha\in R$, $\lambda\in L$, $f\in\mathbb{F}$, and $g\in\mathbb{F}^{\times}$. One may view the elements $x_{\alpha}(f)$ as analogues of the elementary matrices from type~$A$, and the elements $h_{\lambda}(g)$ are analogues of the diagonal matrices. Then the relations in $G$ are the analogues of the usual row reduction operations.  

Let $\theta:\fo\to k$ be the canonical homomorphism onto the residue field $k=\mathfrak{o}/\varpi\mathfrak{o}$ (for example, if $\FF=\FF_q(\!(t)\!)$ then $\theta$ is evaluation at~$t=0$). The \textit{standard Iwahori} subgroup is defined by the following diagram, where $B(k)$ is the standard \textit{Borel subgroup} of $G(k)$.
$$
\begin{array}{cclcl}
G&=&G(\mathbb{F})\\
\rotatebox{90}{$\subseteq$}& &\\
K&=&G(\mathfrak{o})&\xrightarrow{\quad\theta\quad}&G(k)\\
\rotatebox{90}{$\subseteq$}& & & &\,\,\,\rotatebox{90}{$\subseteq$}\\
I&=&\theta^{-1}(B(k))&\xrightarrow{\quad\theta\quad}&B(k)
\end{array}
$$
Then $G/I$ is the set of chambers of an affine building $\Delta$ of type~$W$, where $W$ is the affine Weyl group of~$R$. Moreover, $G/K$ is a subset of the vertex set of~$\Delta$. Specifically, $G/K$ is the set of all vertices $x$ of $\Delta$ with $\tau(x)\in\{\tau(\lambda)\mid\lambda\in L\}$ (if $L=Q$ then this is precisely the set of all type~$0$ vertices of~$\Delta$, and if $L=P$ then it is the set of all \textit{special vertices} of~$\Delta$). The Cartan and Iwasawa decompositions~(\ref{eq:CIdecomp}) hold (with $U$ being the subgroup of $G$ generated by the elements $x_{\alpha}(f)$ with $\alpha\in R^+$ and $f\in\mathbb{F}$, and with $t_{\lambda}$ given by $t_{\lambda}=h_{\lambda}(\varpi^{-1})$). Moreover the vector distance function~$\vect{d}$ and Buesmann function~$\vect{h}$ are exactly as in (\ref{eq:vbuild}) and~(\ref{eq:hbuild}).

\section{Regular sequences in affine buildings}\label{sect:regularity}

Let $\Delta$ be an irreducible affine building with Coxeter complex $\Sigma$ in the 
Euclidean space $E$. Let $\lambda\in E^+$. A \textit{$\lambda$-ray} in $\Delta$ is a 
function $\fr:[0,\infty)\to \Delta$ such that
$$
\vect{d}\bigl(\fr(t_1),\fr(t_2)\bigr)=(t_2-t_1)\lambda\qquad\textrm{for all}\;t_2\geq t_1\geq 0.
$$
Since we are specifying both speed and direction, the notion of a $\lambda$-ray is a refinement 
of the usual notion of a ray in a $\mathrm{CAT}(0)$ space. 

\begin{defn}
A sequence $(x_n)_{n\geq 0}$ is \textit{$\lambda$-regular} if there exists a $\lambda$-ray 
$\fr:[0,\infty)\to \Delta$ such that
$$
d\bigl(x_n,\fr(n)\bigr)=o(n).
$$
Here, $o(n)$ is the usual ``little-$o$'' notation.
Without loss of generality we may stipulate that $\fr(0)=x_0=o$ in the definition, and we will 
do so throughout.
\end{defn}

The main result of this paper is the following characterisation of $\lambda$-regular sequences 
in the spirit of Kaimanovich's original characterisation \cite[Theorems~2.1 and~2.4]{kaimanovich} 
for symmetric spaces.

\newpage

\begin{thm}\label{thm:main}
Let $(x_n)_{n\geq 0}$ be a sequence in $\Delta$, and let $\lambda\in E^+$. Let $\fs$ be a sector of $\Delta$. The following are equivalent.
\begin{enumerate}
\item[\rm (1)] The sequence $(x_n)_{n\geq 0}$ is $\lambda$-regular.
\item[\rm (2)] $d(x_n,x_{n+1})=o(n)\;$ and $\;\vect{h}_{\fs}(x_n)=n\mu_{\fs}+o(n)\;$ 
for some $\mu_{\fs}\in W_0\lambda$ (independent of~$n$).
\item[\rm (3)] $d(x_n,x_{n+1})=o(n)\;$ and $\;\vect{d}(o,x_n)=n\lambda+o(n)$.
\end{enumerate}
\end{thm}

We shall show in several steps that (1) $\Rightarrow$ (2) $\Rightarrow$ (3) $\Rightarrow$ (1). 
The implication $(1)\,\Rightarrow\,(2)$ is quite straightforward:

\begin{proof}[Proof of Theorem~\ref{thm:main}, {\rm (1)}\,$\Rightarrow$\,{\rm (2)}] 

Let $\fr:[0,\infty)\to \Delta$ be a $\lambda$-ray with $d\bigl(x_n,\fr(n)\bigr)=o(n)$. Then
$$
d(x_n,x_{n+1})\leq d\bigl(x_n,\fr(n)\bigr)+d\bigl(\fr(n),\fr(n+1)\bigr)+
d\bigl(\fr(n+1),x_{n+1}\bigr)=o(n).
$$
By \cite[Theorem~11.53]{AB}, the image of the $\lambda$-ray $\fr$ lies in an apartment, 
and hence in a sector~$\ft$. By (S2) there exist a subsector 
$\ft'$ of $\ft$ and a subsector~$\fs'$ of~$\fs$ such that $\ft'$ and $\fs'$ lie in a 
common apartment~$A'$, say. We claim that
\begin{align}\label{eq:isor}
\vect{h}_{\fs}|_{A'}:A'\to E\qquad\textrm{is an isomorphism preserving vector distances}.
\end{align}
Let $A$ be an apartment containing~$\fs$. By definition, 
$\vect{h}_{\fs}=\psi_{A,\fs}\circ \rho_{A,\fs}$, where $\psi_{A,\fs}:A\to E$ is the unique 
isomorphism mapping $\fs$ to $\fs_0$ and preserving vector distances. From the definition 
of sector retractions we have $\rho_{A,\fs}=\rho_{A,\fs'}$ (since $\fs'$ is a subsector of 
$\fs$), and that $\rho_{A,\fs'}|_{A'}:A'\to A$ is the isomorphism fixing $A\cap A'$ pointwise, 
whence~(\ref{eq:isor}).

Let~$z$ be the base vertex of $\ft'$, and let~$\fr'$ be the unique $\lambda$-ray with image 
in $\ft'$ and with $\fr'(0)=z$. Then $d\bigl(\fr(t),\fr'(t)\bigr)$ is a constant (as the 
rays are parallel and in a common apartment; indeed they are both in the sector~$\ft$). 
Since $\fr'$ lies in $A'$,  (\ref{eq:isor}) gives
$$
\vect{h}_{\fs}\bigl(\fr'(t)\bigr)=\vect{h}_{\fs}(z)+t\mu\qquad\textrm{for some} \;\mu\in W_0\lambda.
$$
Since $d\bigl(\fr(t),\fr'(t)\bigr)=O(1)$, and since retractions do not increase distance 
\cite[Theorem~11.16]{AB}, we have 
$\vect{h}_{\fs}\bigl(\fr(n)\bigr)=\vect{h}_{\fs}\bigl(\fr'(n)\bigr)+O(1)=n\mu+O(1)$, and hence 
$\vect{h}_{\fs}(x_n)=n\mu+o(n)$.
\end{proof}

We now turn to the implication (2) $\Rightarrow$ (3). The following lemma (cf. \cite{sph}) is of 
independent interest. It gives a quantitative version of the fundamental result 
\cite[Theorem~11.63]{AB}.

\begin{lemma}\label{lem:crucial} Let $A$ be an apartment of $\Delta$, and let $\fs$ be a sector 
of $A$. For each wall $H$ of $A$, let $H^+$ be the half apartment of $A$ bounded by $H$ and 
containing a subsector of $\fs$. Let $c_0,\ldots,c_n$ be a gallery in $\Delta$ with 
$c_0\subset A$. Let $H_k$ be the wall of $A$ containing the panel $\rho_{A,\fs}(c_{k-1}\cap c_k)$. 
Then there exists an apartment $A_n$ containing $c_n$ and $H_1^+\cap\cdots\cap H_n^+\,$.
\end{lemma}

\begin{proof} We use induction on $n$, with $A_0=A$ starting the induction. 
Suppose that $A_{n-1}$ contains $c_{n-1}$ and $H_1^+\cap\cdots\cap H_{n-1}^+$. 
Let $H$ be the wall of $A_{n-1}$ containing $c_{n-1}\cap c_n$, and let $H^+$ be the half 
apartment of $A_{n-1}$ bounded by $H$ and containing a subsector of $\fs$ 
(note that $A_{n-1}$ contains a subsector of $\fs$ by hypothesis). Let $A_n$ be an apartment 
containing $H^+$ and $c_n$ (Lemma~\ref{lem:config}). We claim that $A_n$ contains 
$H_1^+\cap\cdots\cap H_n^+$. To see this, let $\psi:A\to A_{n-1}$ be an isomorphism fixing 
$A\cap A_{n-1}$ pointwise (this intersection is nonempty since both apartments contain 
a subsector of~$\fs$). Then $\psi(H_n^+)=H^+$, because $\psi$ is the inverse of the 
isomorphism $\rho_{A,\fs}|_{A_{n-1}}:A_{n-1}\to A$ (by the definition of sector based 
retractions). Therefore
\begin{align*}
H_1^+\cap\cdots\cap H_n^+&=\psi(H_1^+\cap\cdots\cap H_n^+)
&&\textrm{(since $H_1^+\cap\cdots\cap H_n^+\subseteq A\cap A_{n-1}$)}\\
&=\psi(H_1^+\cap\cdots\cap H_{n-1}^+)\cap \psi(H_n^+)\\
&=H_1^+\cap\cdots\cap H_{n-1}^+\cap H^+
&&\textrm{(since $H_1^+\cap\cdots\cap H_{n-1}^+\subseteq A_{n-1}$)}
\end{align*}
Thus $H_1^+\cap\cdots\cap H_n^+\subseteq A_n$, completing the proof.
\end{proof}

\begin{lemma}\label{lem:sector2} Let $A$ be an apartment containing a sector $\fs$. 
Let $\fs'$ be a sector of $A$ adjacent to $\fs$, and let $H$ be the wall of $A$ 
separating $\fs$ and $\fs'$. Let $x\in\Delta$, and suppose that there is an apartment 
containing $\fs$ and~$x$. Then
$$
d\bigl(\rho_{A,\fs}(x),\rho_{A,\fs'}(x)\bigr)\leq 2d\bigl(\rho_{A,\fs}(x),H\bigr).
$$
\end{lemma}

\begin{proof}
Let $A'$ be an apartment containing $x$ and $\fs$ (such an apartment exists by hypothesis). 
Let $\psi:A'\to A$ be the isomorphism from (B2) fixing $A'\cap A$ pointwise. Let $H'$ be the wall of 
$A'$ with $\psi(H')=H$. Let $H'_+$ be the half apartment of $A'$ containing $\fs$. By 
Lemma~\ref{lem:config} there exists an apartment $A''$ containing $H'_+\cup \fs'$.

Let $z\in H'$ be a point with $d(x,z)=d(x,H')$ (recall that $x$ and $H'$ lie in the 
apartment $A'$). Then $\rho_{A,\fs}(z)=\rho_{A,\fs'}(z)$ (as $z\in A''$), and
\begin{align*}
d\bigl(\rho_{A,\fs}(x),\rho_{A,\fs'}(x)\bigr)
&\leq d\bigl(\rho_{A,\fs}(x),\rho_{A,\fs}(z)\bigr)+d\bigl(\rho_{A,\fs}(z),\rho_{A,\fs'}(x)\bigr)\\
&=d\bigl(\rho_{A,\fs}(x),\rho_{A,\fs}(z)\bigr)+d\bigl(\rho_{A,\fs'}(z),\rho_{A,\fs'}(x)\bigr)\\
&\leq d(x,z)+d(z,x)=2d(x,z)\qquad\textrm{(\cite[Theorem~11.16]{AB}).}
\end{align*}
But $d(x,z)=d(x,H')=d\bigl(\rho_{A,\fs}(x),H\bigr)$ since $H=\psi(H')$ and 
$\rho_{A,\fs}(y)=\psi(y)$ for all $y\in A'$.\end{proof}

\begin{proof}[Proof of Theorem~\ref{thm:main}, {\rm (2)}\,$\Rightarrow$\,{\rm (3)}] 

It is logically sufficient to prove this implication under the assumption that $\fs=\fs_0$, 
and hence $\vect{h}_{\fs}=\vect{h}$. For each $v\in W_0$ define $\vect{h}_v:\Delta\to E$ by 
$
\vect{h}_v(x)=\rho_{\Sigma,v\fs_0}(x).
$
These are the retraction values from the $|W_0|$ possible sector retraction directions 
in the base apartment. Note that $\vect{h}_{1}=\vect{h}$, and in the notation of 
Definition~\ref{defn:horof}, $\vect{h}_v=\vect{h}_{v\fs_0}$.

Suppose that $(x_n)_{n\geq0}$ is a sequence in $\Delta$ with $d(x_n,x_{n+1})=o(n)$ and 
$\vect{h}(x_n)=n\mu+o(n)$. Let $w\in W_0$ be minimal length subject to $\mu\in w\fs_0$. 
Let $ww_0=s_{i_1}\cdots s_{i_{\ell}}$ be a reduced expression, and for each 
$0\leq k\leq \ell$ let $v_k=s_{i_1}\cdots s_{i_k}$ (so that $v_0=1$ and $v_{\ell}=ww_0$). 
We will prove the following.
\smallskip

\noindent\textit{Claim:} For each $0\leq k\leq \ell$ we have $\vect{h}_{v_k}(x_n)=n\mu+o(n)$, 
and there is an apartment $A_k$ containing $x_n$ and the sector $\ft_k=z_k+v_k\fs_0$ of $\Sigma$ 
based at a point
$$
z_k=\sum_{\{j\mid \langle\mu,v_k\alpha_j\rangle>0\}}\langle n\mu,v_k\alpha_j\rangle v_k\omega_j+o(n).
$$

\smallskip

Given the claim, we have $\vect{h}_{ww_0}(x_n)=n\mu+o(n)$ and there is an apartment $A_{\ell}$ 
containing $x_n$ and the sector $\ft_{\ell}=z_{\ell}+ww_0\fs_0$ of~$\Sigma$. Thus, by the 
definition of $\vect{h}_{ww_0}\,$, if $\psi:A_{\ell}\to\Sigma$ is the isomorphism fixing 
$A_{\ell}\cap\Sigma$ pointwise, then $\psi(x_n)=\vect{h}_{ww_0}(x_n)=n\mu+o(n)$, 
and $\psi(z_{\ell})=z_{\ell}$. Note that $z_{\ell}=o(n)$ because 
$
\langle\mu,v_{\ell}\alpha_j\rangle=\langle w^{-1}\mu,w_0\alpha_j\rangle\leq0
$
(since $w_0\alpha_j\in -R^+$ and $w^{-1}\mu\in \fs_0$). Hence by the definition of vector 
distances we have
$$
\vect{d}(z_{\ell},x_n)=\bigl(\psi(x_n)-\psi(z_{\ell})\bigr)^+=\bigl(n\mu+o(n)\bigr)^+
=n\lambda+o(n),
$$
where $\lambda$ is the dominant element of $W_0\mu$.
Therefore $\vect{d}(o,x_n)=n\lambda+o(n)$, since $\vect{d}(o,z_{\ell})=o(n)$. 
It remains to prove the claim.

\smallskip

\noindent\textit{Proof of the claim.} We argue by induction. For $k=0$, construct a 
gallery $\gamma$ from $o=x_0$ to $x_n$ by picking minimal galleries joining $x_{i-1}$ to 
$x_i$ ($1\leq i\leq n$) and joining them together. Since $\vect{h}(x_n)=n\mu+o(n)$ and 
$d(x_n,x_{n+1})=o(n)$, the image of this gallery under $\vect{h}$ cannot deviate too 
far from the geodesic segment $[0,n\mu]$ in $\Sigma$. In particular, the image of 
$\gamma$ can only cross hyperplanes $H_{\alpha,j}$ with $\alpha\in R^+$ and 
\begin{align}\label{eq:hyperplanes}
\begin{cases}
-o(n)\leq j\leq n\langle\mu,\alpha\rangle+o(n)&\textrm{if $\langle\mu,\alpha\rangle\geq0$}\\
n\langle\mu,\alpha\rangle-o(n)\leq j\leq o(n)&\textrm{if $\langle\mu,\alpha\rangle<0$}.
\end{cases}
\end{align}
Then by Lemma~\ref{lem:crucial} there is an apartment $A_0$ containing $\ft_0=z_0+\fs_0$ and 
$x_n\,$, because the point 
$
z_0=\sum_{\{j\mid \langle\mu,\alpha_j\rangle>0\}}\langle n\mu,\alpha_j\rangle\omega_j+o(n)
$ 
is on the positive side of all of the hyperplanes~(\ref{eq:hyperplanes}).

Suppose that the claim is true for some $k$ with $k<\ell$. Let $H$ be the hyperplane 
of~$\Sigma$ separating the sectors $\ft_k=z_k+v_k\fs_0$ and $\ft_k'=z_k+v_{k+1}\fs_0$ 
(these sectors are adjacent since they are translates of the adjacent sectors $v_k\fs_0$ 
and $v_{k+1}\fs_0$). By the induction hypothesis there is an apartment $A_k$ containing 
$\ft_k$ and $x_n\,$, and so by Lemma~\ref{lem:sector2} we have
\begin{align}\label{eq:cl}
d\bigl(\vect{h}_{v_k}(x_n),\vect{h}_{v_{k+1}}(x_n)\bigr)
\leq 2d\bigl(\vect{h}_{v_k}(x_n),H\bigr)=2d(n\mu,H)+o(n)
\end{align}
(since $\vect{h}_{v_k}(x_n)=n\mu+o(n)$ by the induction hypothesis).   
We have $H=z_k+H_{v_k\alpha_{i_{k+1}}}$, and so by the usual perpendicular distance formula
\begin{align}\label{eq:distform}
d(n\mu,H)=\frac{|\langle n\mu-z_k,v_k\alpha_{i_{k+1}}\rangle|}{\|v_k\alpha_{i_{k+1}}\|}.
\end{align}
Note that $\mu$ is on the positive side of each of the hyperplanes $H_{\alpha}$ with 
$\alpha$ in the \textit{inversion set} 
$R(ww_0)=\{\alpha\in R^+\mid (ww_0)^{-1}\alpha\in -R^+\}$, and that this inversion set 
is given explicitly by 
$R(ww_0)=\{\alpha_{i_1},v_1\alpha_{i_2},\ldots,v_{\ell-1}\alpha_{i_{\ell}}\}$ 
(see \cite[VI, \S6, Cor~2]{bourbaki}). Thus $\langle n\mu,v_k\alpha_{i_{k+1}}\rangle>0$, 
and so from the definition of $z_k$ we see that 
$\langle z_k,v_k\alpha_{i_{k+1}}\rangle=\langle n\mu,v_k\alpha_{i_{k+1}}\rangle+o(n)$. 
Therefore~(\ref{eq:distform}) gives $d(n\mu,H)=o(n)$, and so (\ref{eq:cl}) gives 
$\vect{h}_{v_{k+1}}(x_n)=\vect{h}_{v_k}(x_n)+o(n)=n\mu+o(n)$. 

Now repeat the $k=0$ argument for the new retraction direction $v_{k+1}\fs_0$. The image of 
the gallery $\gamma$ under $\vect{h}_{v_{k+1}}$ only crosses hyperplanes of the 
form (\ref{eq:hyperplanes}), and we see by Lemma~\ref{lem:crucial} that there is an 
apartment $A_{k+1}$ containing $x_n$ and the sector $\ft_{k+1}=z_{k+1}+v_{k+1}\fs_{0}$. 
This completes the  proof of the claim.
\end{proof}

We now turn to the final implication (3) $\Rightarrow$ (1).

\begin{lemma}\label{lem:ineq} Let $o,a,b\in \Delta$ (in fact in any $\mathrm{CAT}(0)$ space). 
Let $t_1,t_2\in[0,1]$. Let $p(t_1)\in[o,a]$ with $d\bigl(o,p(t_1)\bigr)=t_1\,d(o,a)$ and let 
$q(t_2)\in[o,b]$ with $d\bigl(o,q(t_2)\bigr)=t_2\, d(o,b)$. Then
$$
d^2\bigl(p(t_1),q(t_2)\bigr)\leq t_1(t_1-t_2)\,d^2(o,a)+t_2(t_2-t_1)\,d^2(o,b)+t_1t_2\,d^2(a,b).
$$
\end{lemma}

\begin{proof} By (\ref{eq:negcurv}) with $(x,y,z)=\bigl(o,b,p(t_1)\bigr)$ we have
\begin{align*}
d^2\bigl(p(t_1),q(t_2)\bigr)\leq(1-t_2)\,d^2\bigl(p(t_1),o\bigr)
+t_2\, d^2\bigl(p(t_1),b\bigl)-t_2(1-t_2)\,d^2(o,b).
\end{align*}
Apply (\ref{eq:negcurv}) to $d^2\bigl(p(t_1),o\bigr)$ with $(x,y,z)=(o,a,o)$ and to 
$d^2\bigl(p(t_1),b\bigl)$ with $(x,y,z)=(o,a,b)$.
\end{proof}

\begin{lemma}\label{lem:h}
Suppose that $x,y,z\in \Delta$ with $\lambda=\vect{d}(z,x)=\vect{d}(z,y)$ and 
$[z,x]\cap[z,y]=\{z\}$. There is $C>0$ depending only on the direction of $\lambda$ 
(and not on its length) such that
$
d(x,y)\geq C \, d(z,x).
$
\end{lemma}

\begin{proof}
Choose a chamber $c$ containing a non-zero length initial piece of the geodesic $[z,x]$, 
and similarly let $d$ be a chamber containing a non-zero length initial piece of the 
geodesic $[z,y]$ (with the possibility that $c=d$). Let $A$ be an apartment containing 
$c$ and $d$ (using (B1)). Since $z\in c$, the (chamber based) retraction $\rho=\rho_{A,c}$ maps the 
geodesics $[z,x]$ and $[z,y]$ to geodesics in~$A$, and so 
$\vect{d}\bigl(z,\rho(x)\bigr)=\vect{d}\bigl(z,\rho(y)\bigr)=\lambda$. Furthermore, 
since $c,d\in A$, the hypothesis $[z,x]\cap[z,y]=\{z\}$ implies that the images of the 
geodesics $[z,x]$ and $[z,y]$ under $\rho$ are not equal. 

Let $\psi:A\to \Sigma$ be a type preserving isomorphism, and let $\mu=-\psi(z)$. 
Let $\theta=t_{\mu} \psi:A\to\Sigma$. From the definition of vector distances we have 
$\theta\bigl(\rho(x)\bigr)=w_1\lambda$ and $\theta\bigl(\rho(y)\bigr)=w_2\lambda$, 
and by the above observation $w_1\lambda\neq w_2\lambda$. Then (using \cite[Theorem~11.16]{AB})
$$
d(x,y)\geq d\bigl(\rho(x),\rho(y)\bigr)
=d\Bigl(\theta\bigl(\rho(x)\bigr),\theta\bigl(\rho(y)\bigr)\Bigr)=d(w_1\lambda,w_2\lambda).
$$
Euclidean geometry and the fact that $\|w\lambda\|=\|\lambda\|$ gives 
$
d(w_1\lambda,w_2\lambda)=2\|\lambda\|\sin(\theta/2),
$
where $\theta$ is the angle between $w_1\lambda$ and $w_2\lambda$. But $\|\lambda\|=d(z,x)$, and 
the result follows since $w_1\lambda\neq w_2\lambda$ and $W_0$ is finite.
\end{proof}

\begin{proof}[Proof of Theorem~\ref{thm:main}, {\rm (3)}\,$\Rightarrow$\,{\rm (1)}]
Let $(x_n)_{n\geq 0}$ be a sequence in $\Delta$ with $d(x_n,x_{n+1})=o(n)$ and 
$\vect{d}(o,x_n)=n\lambda+o(n)$, with $\lambda\in E^+$. If $\lambda=0$ then 
$(x_n)_{n\geq0}$ is a $0$-regular sequence, and we are done. So suppose that $\lambda\neq 0$. 
It is clear that there is a sequence $(y_n)_{n\geq 0}$ such that $d(x_n,y_n)=o(n)$ and 
$\vect{d}(o,y_n)=n\lambda$. We show that $(y_n)_{n\geq0}$ is $\lambda$-regular, and 
thus $(x_n)_{n\geq 0}$ is $\lambda$-regular, too.

First we construct a $\lambda$-ray $\fr:[0,\infty)\to \Delta$. Let $t\geq 0$ be fixed. 
Since $\la\neq 0$ there is $N_t>0$ such that $d(o,y_n)>t\|\lambda\|$ for all $n>N_t$. If $n>N_t$ 
let $y_n(t)$ be the point on the geodesic $[o,y_n]$ with $d\bigl(o,y_n(t)\bigr)=t\|\lambda\|$. 
Lemma~\ref{lem:ineq} with $(o,a,b)=(o,y_n,y_{n+1})$ and $t_1=t\|\lambda\|/|y_n|$ and
$t_2= t\|\lambda\|/|y_{n+1}|)$ (where $|x|=d(o,x)$) yields
$$
d^2\bigl(y_n(t),y_{n+1}(t)\bigr)\leq t^2\|\lambda\|^2
\left(\frac{d^2(y_n,y_{n+1})}{|y_n|\,|y_{n+1}|}
-\left(\frac{|y_n|}{|y_{n+1}|}-\frac{|y_{n+1}|}{|y_n|}\right)^2\right).
$$
Since $d(y_n,y_{n+1})=o(n)$ and $|y_n|=n\|\la\|$ it follows that
$$
\lim_{n\to\infty}d\bigl(y_n(t),y_{n+1}(t)\bigr)=0\qquad\textrm{for each fixed $t\geq0$}.
$$
But $y_j(t)$ is a point in the building with $\vect{d}(o,y_j(t))=t\lambda$, and so the set $\{\vect{d}(y_i(t),y_j(t))\mid i,j\geq N_t\}$ is a finite set. Thus there 
is an index $N_t'>N_t$ such that $y_n(t)=y_m(t)$ for all $m,n>N_t'$. Denote this 
stabilised point by $\fr(t)$. Then $\fr(t)\in[o,y_n]$ for all $n>N_t'$. This defines a 
$\lambda$-ray $\fr:[0,\infty)\to \Delta$.

Next we show that $d\bigl(y_n,\fr(n)\bigr)=o(n)$. Let $T$ be the (metric) tree
$$
T=\bigcup_{n\geq 0}[o,y_n]\subseteq \Delta.
$$
The \textit{confluent} $x\curlywedge y$ of $x,y\in T$ is defined by $
[o,x]\cap [o,y]=[o,x\curlywedge y]
$ (note that $x\curlywedge y\in T$). 
Define a metric $d_T$ on $T$ by
$
d_T(x,y)=d(x,x\curlywedge y)+d(x\curlywedge y,y)$ for all $x,y\in T$. Then we have
\begin{align}\label{eq:confluent}
d_T(o,x\curlywedge y)=\frac{1}{2}\Bigl(d_T(o,x)+d_T(o,y)-d_T(x,y)\Bigr)\quad\textrm{for $x,y\in T$}.
\end{align}
This quantity is often called the \textit{Gromov product} of $x$ and $y$ at~$o$. We have $d_T(o,x)=d(o,x)$ for all $x\in T$, and we will show below that 
$d_T(y_n,y_{n+1})=o(n)$ (initially it is only clear that $d_T(y_n,y_{n+1})\geq d(y_n,y_{n+1})=o(n)$).
Assuming this for now, equation (\ref{eq:confluent}) with $x=y_n$ and $y=y_{n+1}$ gives
$d(o,y_n\curlywedge y_{n+1})=n\|\lambda\|+o(n).$

Note that $T$ contains (the image of) the $\lambda$-ray $\fr$. Define the confluent 
$x\curlywedge \fr$ of $x\in T$ with $\fr$ by the equation $[o,x]\cap \fr=[o,x\curlywedge \fr]$. 
Then, exactly as in \cite[\S~2C]{CKW}, we have
\begin{align*}
n\|\lambda\|=|y_n|\geq d(o,y_n\curlywedge \fr)
&=\lim_{m\to\infty}d(o,y_n\curlywedge y_m)\geq
\inf_{i\geq n}d(o,y_i\curlywedge y_{i+1})=n\|\lambda\|+o(n).
\end{align*}
Thus $d(o,y_n\curlywedge \fr)=n\|\lambda\|+o(n)$, and so
$$
d\bigl(y_n,\fr(n)\bigr)\leq 
d\bigl(y_n,y_n\curlywedge \fr\bigr)+d\bigl(y_n\curlywedge \fr,\fr(n)\bigr)=o(n).
$$
Hence $(y_n)_{n\geq 0}$ is $\lambda$-regular.

It remains to show that $d_T(y_n,y_{n+1})=o(n)$. It suffices to show that 
$d(y_n\curlywedge y_{n+1},y_n)=o(n)$. Let $y_n'\in[o,y_{n+1}]$ be the point with 
$\vect{d}(o,y_n')=n\lambda$. Then $d(y_n\curlywedge y_n',y_n)=d(y_n\curlywedge y_{n+1},y_n)+O(1)$.
By Lemma~\ref{lem:h} (with $z=y_n\curlywedge y_n'$, $x=y_n$, and $y=y_n'$) we have
$
d(y_n,y_n')\geq C\,d(y_n\curlywedge y_n',y_n)
$
for some constant $C>0$ depending only on the direction of $\lambda$. 
Hence $d(y_n\curlywedge y_{n+1},y_n)=o(n)$, completing the proof.
\end{proof}

\section{Applications}\label{sect:applications}

In this section we exhibit applications of Theorem~\ref{thm:main} 
to random walks on groups acting on affine buildings, and to random walks on the 
buildings themselves (where there may be no group present).

\subsection{Random walks on groups acting on affine buildings}

Let $\Aut(\Delta)$ be the group of simplicial complex automorphisms of an affine 
building~$\Delta$ equipped with the topology of pointwise convergence. 
Thus $\Aut(\Delta)$ is a totally disconnected locally compact Hausdorff group, 
with a neighbourhood base at the identity given by the family of all pointwise 
stabilisers of finite sets of vertices. 

Let $G$ be a subgroup of $\Aut(\Delta)$, and let 
$\sigma$ be a Borel probability measure on~$G$. 
We will assume that the support of $\sigma$ generates $G$, and we say that $\sigma$ has 
\textit{finite first moment} if
$$
\int_Gd(o,go)\,d\sigma(g)<\infty.
$$
Let $(g_n)_{n\geq 1}$ be a stationary sequence of $G$-valued random variables with 
joint distribution~$\sigma$. The \textit{right random walk} is the sequence $(X_n)_{n\geq 0}$ with 
$$
X_0=o\qquad\textrm{and}\qquad X_n=g_1\ldots g_no\quad\textrm{for $n\geq 1$}.
$$
It follows from \cite[Theorem~2.1]{KM} that if 
$\sigma$ has finite first moment then there is a unit speed geodesic $\gamma:[0,\infty)\to\Delta$ 
and a number $a\geq 0$ such that 
\begin{align}\label{eq:boundary}
\lim_{n\to\infty}\frac{1}{n}d\bigl(X_n,\gamma(an)\bigr)=0\qquad\textrm{almost surely}.
\end{align}
Therefore $(X_n)_{n\geq 0}$ is almost surely $\lambda$-regular, where 
$\lambda=\vect{d}\bigl(\gamma(0),\gamma(a)\bigr)$. Thus Theorem~\ref{thm:main} immediately gives
the following.

\begin{thm}\label{thm:4.1}
Let $G$ and $\sigma$ be as above, and suppose that $\sigma$ has finite first moment. 
Let $(X_n)_{n\geq 0}$ be the associated right random walk on~$\Delta$. Then there exists 
$\lambda\in E^+$ such that
$$
\lim_{n\to\infty}\frac{1}{n}\vect{d}(o,X_n)=\lambda\qquad\textrm{almost surely},
$$
and for each sector $\fs$ of $\Delta$ there exists $\mu_{\fs}\in W_0\lambda$ such that 
$$
\lim_{n\to\infty}\frac{1}{n}\vect{h}_{\fs}(X_n)=\mu_{\fs}\qquad\textrm{almost surely}.
$$
\end{thm}

\begin{cor}\label{cor:det}
Let $\fs$ be a sector of $\Delta$, and let $G$, $\sigma$, $(X_n)_{n\geq 0}$, $\lambda$, and $\mu_{\fs}$ be as in Theorem~\ref{thm:4.1} 

If $\lambda\neq 0$ then $(X_n)_{n\geq 0}$ converges almost surely to a point $X_{\infty}$ 
of the visibility boundary. Moreover, if $\mu_{\fs}=\lambda\neq 0$ is dominant then $X_{\infty}$ is the equivalence class of any 
$\lambda$-ray contained in~$\fs$ (and is hence deterministic).
\end{cor}

\begin{proof}
If $\lambda\neq 0$ then convergence to a point of the 
visibility boundary is an immediate consequence of~(\ref{eq:boundary}). Suppose 
that $\mu_{\fs}=\lambda\neq 0$ is dominant. Let $\fr:[0,\infty)\to \Delta$ be a 
$\lambda$-ray such that $d\bigl(X_n,\fr(n)\bigr)=o(n)$ almost surely, and without 
loss of generality we may assume that $\fr(0)$ is the base vertex of~$\fs$. We 
claim that $\fr(t)=\fs(t\lambda)$ for all $t\geq 0$, where for $\nu\in E^+$ we 
write $\fs(\nu)$ for the unique point of $\fs$ with 
$\vect{d}\bigl(\fs(0),\fs(\nu)\bigr)=\nu$ (with $\fs(0)$ being the base point of $\fs$).

Since $d\bigl(X_n,\fr(n)\bigr)=o(n)$ and $\vect{h}_{\fs}(X_n)=n\lambda+o(n)$ almost surely, 
we have $\vect{h}_{\fs}\bigl(\fr(n)\bigr)=n\lambda+o(n)$. Lemma~\ref{lem:crucial} 
implies that there is an apartment containing the point $\fr(n)$ and the subsector of 
$\fs$ based at a point $\fs\bigl(n\lambda+o(n)\bigr)$. 
It follows that $d\bigl(\fr(n),\fs(n\lambda)\bigl)=o(n)$. But $\{\fr(t)\mid t\geq 0\}$ and 
$\{\fs(t\lambda)\mid t\geq 0\}$ are two rays orginiating at the same point, and hence 
$\fr(n)=\fs(n\lambda)$, see \cite[Proposition~II.8.2]{bridson} or \cite[Lemma~11.72]{AB}.
\end{proof}

\begin{cor}\label{cor:4.3}
Let $R$ be a reduced irreducible root system with coroot lattice $Q$ and coweight lattice~$P$, and let $L$ be a lattice with $Q\subseteq L\subseteq P$. Let $\mathbb{F}$ be a non-archimedean local field, and let $G=G(\mathbb{F})$ be the Chevalley group over~$\mathbb{F}$ with root datum $(R,L)$, as in Section~\emph{\ref{sect:ex}}. Let $o=K$ and let $g_1,g_2,\ldots$ be a stationary sequence in~$G$ with finite first moment.
\begin{itemize}
\item[\emph{(1)}] There are elements $\lambda_n\in L\cap E^+$ and $\mu_n\in L$ such that $g_1\cdots g_no\in Kt_{\lambda_n}K\cap Ut_{\mu_n}K$.
\item[\emph{(2)}] There exists $\lambda\in E^+$ and $w\in W_0$ such that $\lambda_n/n\to \lambda$ and $\mu_n/n\to w\lambda$ almost surely. 
\end{itemize}
\end{cor}

\begin{proof}
The first statement follows from the Cartan and Iwasawa decompositions~(\ref{eq:CIdecomp}). Then, from (\ref{eq:vbuild}) and (\ref{eq:hbuild}) we have $\vect{d}(o,g_1\cdots g_no)=\lambda_n$ and $\vect{h}(g_1\cdots g_no)=\mu_n$, and the result follows from Theorem~\ref{thm:4.1}.
\end{proof}

\begin{remark} The \textit{drift-free case} when $\lambda=0$ is rather subtle, even in 
the rank~1 case of homogeneous trees (see \cite{CKW,SB}). We discuss this case 
further in Section~\ref{sect:driftfree}.
\end{remark}

\subsection{Semi-isotropic random walks on affine buildings}

Let $\Delta$ be a locally finite irreducible affine building of type $(W,S)$ with vertex 
set~$V$. Let $L$ be a lattice with $Q\subseteq L\subseteq P$, and let
$$
V_L=\{v\in V\mid\tau(v)\in\tau(L)\},\qquad\textrm{where}\qquad 
\tau(L)=\{\tau(\lambda)\mid \lambda\in L\}.
$$
Then $V_Q$ is the set of all type $0$ vertices of $\Delta$, and $V_P$ is the set of 
so called \textit{special vertices} of $\Delta$ (in type $\tilde{A}_r$ one has $V_P=V$; 
otherwise $V_P$ is a strict subset of $V$). 

A Markov chain $(X_n)_{n\geq 0}$ on $V_L$ is an \textit{isotropic random walk} if the transition 
probabilities $p(x,y)$ $(x,y\in V_L$) of the random walk depend only on the vector distance 
$\vect{d}(x,y)$. That is,
$p(x,y)=p(x',y')$ whenever $\vect{d}(x,y)=\vect{d}(x',y')$. Isotropic random walks 
can be analysed in great detail by use of harmonic analysis and representation theory, 
see for example \cite{CW,P3,BS}. In particular, precise Local Limit 
Theorems, Central Limit Theorems, and Rate of Escape Theorems are available.

We consider the following more general situation. 

\begin{defn} A Markov chain $(X_n)_{n\geq 0}$ on $V_L$ is a \textit{semi-isotropic random walk} if 
the transition probabilities $p(x,y)$ of the walk depend only on the vectors $\vect{d}(x,y)$ 
and $\vect{h}(y)-\vect{h}(x)$. That is, $p(x,y)=p(x',y')$ whenever 
$\vect{d}(x,y)=\vect{d}(x',y')$ and $\vect{h}(y)-\vect{h}(x)=\vect{h}(y')-\vect{h}(x')$.
\end{defn}

Clearly, isotropic random walks are necessarily semi-isotropic, but not vice versa. 
We will apply Theorem~\ref{thm:main} to prove a Rate of Escape Theorem and 
convergence to the boundary results for semi-isotropic random walks. Note 
that in this setting we cannot apply \cite[Theorem~2.1]{KM} since our random 
walks may not be `group related'. Indeed there are $\tilde{A}_2$, 
$\tilde{C}_2$, and $\tilde{G}_2$ buildings with trivial automorphism group,
see \cite{ronanconstruction}. Even in higher rank, semi-isotropic random walks 
may be unrelated to group walks. 

A building $\Delta$ with chamber set $\cC$ is \textit{locally finite} if the 
cardinality $|\{d\in\cC\mid d\sim_i c\}|$ is finite for all $c\in\cC$ and all 
$i\in I=\{0,\ldots,r\}$, and $\Delta$ is \textit{thick} if this cardinality 
is always at least~$3$ (so that every panel lies on at least $3$ chambers). 
The building $\Delta$ is \textit{regular} if the above cardinality does not 
depend on $c\in\cC$, in which case we define the \textit{parameters} of the building by
$$
q_i+1=|\{d\in\cC\mid d\sim_i c\}|\qquad\textrm{for each $i\in I$}.
$$
By \cite[Corollary~2.2]{P1} we have $q_i=q_j$ whenever $s_i$ and $s_j$ are conjugate 
in~$W$. Furthermore, if the order $m_{ij}$ of $s_is_j$ is finite for all $i,j\in I$ 
then thickness implies regularity, see \cite[Theorem~2.4]{P1}. Thus, all irreducible 
thick affine buildings of rank at least~$3$ are regular.

Let $\Delta$ be an irreducible locally finite regular affine building, and let $L$ 
be a lattice with $Q\subseteq L\subseteq P$. In what follows the lattice $L$ can be 
chosen freely, with two exceptions where we must specify~$L$  precisely. We make 
the following convention in the present sub-section (where we use the standard 
labelling conventions for the simple roots of root systems as recorded in~\cite{bourbaki}):
\begin{align}\label{eq:convention}
\textit{If $\Delta$ is of type $\tilde{C}_r$ with $q_{0}\neq q_{r}$ 
or type $\tilde{A}_1$ with $q_{0}\neq q_{1}$, then we fix $L=Q$}.
\end{align}
For example, in the $\tilde{A}_1$ case with $q_0\neq q_1$, the building is a semi-homogeneous 
tree with alternating vertex valencies $q_0+1$ and $q_1+1$. The restriction $L=Q$ 
means that we only look at the vertices with degree $q_0+1$.

\begin{prop}\label{prop:proj} Let $\Delta$ be a locally finite regular irreducible affine 
building, and let $L$ be a lattice with $Q\subseteq L\subseteq P$, 
with the convention~(\ref{eq:convention}) in force. 
Let $(X_n)_{n\geq 0}$ be a semi-isotropic random walk on $V_L$ with transition 
probabilities $p(x,y)$. Then $(X_n)_{n\geq 0}$ is factorisable over $L$ relative to 
the decomposition of $V_L$ into the sets $H_{\mu}=\{x\in V_L\mid \vect{h}(x)=\mu\}$. 
Moreover, the factor walk $\overline{X}_n=\vect{h}(X_n)$ is a translation invariant 
random walk on the lattice~$L$. 

In other words, the value of the sum 
$$
\overline{p}(\lambda,\mu)=\sum_{y\in H_{\mu}}p(x,y)\qquad
\textrm{with $\lambda,\mu\in L$ and $x\in H_{\lambda}$}
$$
does not depend on the particular $x\in H_{\lambda}$, and 
$\overline{p}(\lambda+\nu,\mu+\nu)=\overline{p}(\lambda,\mu)$ for all $\lambda,\mu,\nu\in L$.
\end{prop}

\begin{proof}
Let $L^+=L\cap E^+$. It is shown in \cite[Lemma~3.19]{sph} that if $\nu\in L^+$ and 
$\mu\in L$ then the cardinality
$$
c_{\nu,\mu}=|\{y\in V_L\mid \vect{d}(x,y)=\nu\textrm{ and }\vect{h}(y)-\vect{h}(x)=\mu\}|
$$
is independent of $x\in V_L$ (the convention~(\ref{eq:convention}) is crucial here). 
By the definition of semi-isotropic random walks we can write $p_{\nu,\mu}=p(x,y)$ when 
$\vect{d}(x,y)=\nu$ and $\vect{h}(y)-\vect{h}(x)=\mu$. Then, if $x\in H_{\lambda}$, we have
\begin{align}\label{eq:formt}
\sum_{y\in H_{\mu}}p(x,y)
&=\sum_{\nu\in L^+}\sum_{\;\,\{y\in H_{\mu}\mid \vect{d}(x,y)=\nu\}}p(x,y)
=\sum_{\nu\in L^+}p_{\nu,\mu-\lambda}\, c_{\nu,\mu-\lambda}\,.
\end{align}
Thus the value does not depend on the particular $x\in H_{\lambda}\,$, 
and the translation invariance is also immediate from this formula.
\end{proof}

\begin{thm}\label{thm:roe} Let $\Delta$ be a locally finite regular irreducible affine 
building, and let $L$ be a lattice with $Q\subseteq L\subseteq P$ 
(with the convention~(\ref{eq:convention}) in force). Let $(X_n)_{n\geq 0}$ be a 
semi-isotropic random walk on $V_L$ with transition probabilities $p(x,y)$, satisfying 
the first moment condition
\begin{align*}
\sum_{\nu\in L}\overline{p}(0,\nu)\|\nu\|<\infty.
\end{align*}
Then the sequence $(X_n)_{n\geq 0}$ is almost surely $\lambda$-regular, 
where $\lambda$ is the dominant element of $W_0\mu$, with 
$\mu=\sum_{\nu\in L}\overline{p}(0,\nu)\nu$. 
\end{thm}

\begin{proof}
The factor walk $\overline{X}_n=\vect{h}(X_n)$ is a translation invariant random walk on 
$L\cong \mathbb{Z}^r$ with finite first moment, and so the classical Law of Large Numbers 
applies. Thus
\begin{align}\label{eq:mudrift}
\lim_{n\to\infty}\frac{\vect{h}(X_n)}{n}=\mu
=\sum_{\nu\in L}\overline{p}(0,\nu)\nu\qquad\textrm{almost surely}.
\end{align}
Since $d(X_n,X_{n+1})=o(n)$ we have $\lambda$-regularity by Theorem~\ref{thm:main}.
\end{proof}

The following is now immediate from Theorems~\ref{thm:main} and~\ref{thm:roe} and the argument 
of Corollary~\ref{cor:det}.

\begin{cor} In the set-up of Theorem~\ref{thm:roe}, we have the Rate of Escape Theorem
$$
\lim_{n\to\infty}\frac{1}{n}\vect{d}(o,X_n)=\lambda\qquad\textrm{almost surely},
$$
where $\lambda$ is the dominant element of the $W_0$-orbit of 
$\mu=\sum_{\nu\in P}\overline{p}(0,\nu)\nu$.

If $\lambda\neq 0$ then $(X_n)_{n\geq 0}$ converges almost surely to a 
point~$X_{\infty}$ of the visibility boundary, and in the case when $\mu=\lambda\neq 0$ is dominant, the point $X_{\infty}$ is the 
equivalence class of the geodesic $\{t\lambda\mid t\geq 0\}$ in $\Sigma$ (and is 
hence deterministic).
\end{cor}

\subsection{The drift free case for semi-isotropic walks}\label{sect:driftfree}

In this section we investigate convergence properties of drift-free semi-isotropic random 
walks (that is, the subtle situation where the vector $\mu$ from Theorem~\ref{thm:roe} is $0$). 
We restrict to the case of nearest neighbour walks (see the definition below). In this case 
the random walk does not converge in the visibility boundary, however we obtain a weaker 
type of convergence, to an equivalence class of sectors in the \textit{combinatorial boundary} 
(see the definition below).

A random walk on the set $V_P$ of special vertices of an affine building~$\Delta$ is called 
\textit{nearest neighbour} if the transition probabilities of the random walk satisfy:
$$
\textrm{If $p(x,y)>0$ then $x$ and $y$ lie in a common chamber.}
$$
In type $\tilde{A}_r$ all vertices are special, and so a nearest neighbour 
random walk (as defined above) is the same as a nearest neighbour random walk 
(in the usual sense) on the $1$-skeleton of the building. 

\begin{remark}
For nearest neighbour random walks to be non-trivial we need each chamber to have at 
least two special vertices. This occurs if and only if there are at least $2$ different 
types of special vertices (since each chamber has a vertex of each type), and by the 
classification of root systems this occurs for all irreducible root systems other 
than those of type $E_8$, $F_4$, or $G_2$. Thus, the results of this section apply 
to affine buildings whose type is not one of the latter three.
For buildings of type $\tilde{E}_8$, $\tilde{F}_4$, or $\tilde{G}_2$ one could use a 
modified definition of `nearest neighbour' by requiring that if $p(x,y)>0$ then $x$ and 
$y$ lie in adjacent chambers. After slight modifications to the proof, 
the below Theorem~\ref{thm:roe0} also holds for these walks although we omit the details.
\end{remark}

The following proposition shows that nearest neighbour random walks on $V_P$ with 
full support are necessarily irreducible. 
 
\begin{prop}\label{prop:irreduciblewalk}
Let $\Delta$ be a locally finite affine building of type other than $\tilde{E}_8$, $\tilde{F}_4$ or $\tilde{G}_2$, and let $(X_n)_{n\geq 0}$ be a random walk on~$V_P$ with transition probabilities $p(x,y)$. Suppose that $p(x,y)>0$ whenever $x$ and $y$ are distinct special vertices lying in a common chamber. Then $(X_n)_{n\geq0}$ is irreducible on~$V_P$. 
\end{prop}

\begin{proof}
Suppose that $X_0=x$ and let $y\in V_P$ be a special vertex. We show that there exists $n>0$ 
such that $p^{(n)}(x,y)>0$. Let $c$ be a chamber of $\Delta$ containing $x$, and let $d$ 
be a chamber of $\Delta$ containing $y$. We argue by induction on the length of a minimal 
length gallery from~$c$ to $d$. If this distance is~$0$ then $c=d$, and so $p(x,y)>0$ by 
hypothesis. Let $c=c_0\sim c_1\sim\cdots\sim c_k=d$ be a minimal length gallery from $c$ to $d$, 
and let $z$ be a special vertex of $c_{k-1}$. By the induction hypothesis there is an $n>0$ 
such that $p^{(n)}(x,z)>0$. If $z$ and $y$ lie in a common chamber then 
$p^{(n+1)}(x,y)\geq p^{(n)}(x,z)p(z,y)>0$. If $z$ and $y$ do not lie in a common chamber 
then there is a special vertex $z'$ in the panel $c_{k-1}\cap d$ (since we exclude type 
$\tilde{E}_8$, $\tilde{F}_4$, and $\tilde{G}_2$). Then $z$ and $z'$ lie in $c_{k-1}$, and $z'$ and 
$y$ lie in $d$, and so $p^{(n+2)}(x,y)\geq p^{(n)}(x,z)p(z,z')p(z',y)>0$.
\end{proof}

\begin{remark}
Let $\Delta$ be an affine building of type other than $\tilde{E}_8$, $\tilde{F}_4$, or $\tilde{G}_2$. 
Define a graph structure on $V_P$ by declaring distinct vertices $x$ and $y$ to be adjacent if 
and only if they lie in a common chamber. The proof of Proposition~\ref{prop:irreduciblewalk} 
shows that the resulting graph is connected, and `nearest neighbour' walks according to our 
definition are simply nearest neighbour walks on this graph.
In the case that $\Delta$ has type $\tilde{E}_8$, $\tilde{F}_4$, or $\tilde{G}_2$ we can define 
the graph structure on $V_P$ by declaring distinct vertices $x$ and $y$ to be adjacent if and 
only if they lie in adjacent chambers. A similar proof shows that this graph is connected.
\end{remark}

\begin{defn}\label{def:ends}
Two sectors $\fs$ and $\fs'$ of $\Delta$ are \textit{equivalent} if their intersection 
$\fs\cap\fs'$ contains a sector. The set of equivalence classes of sectors is the 
\textit{combinatorial boundary} of $\Delta$, denoted~$\Omega$. We call the elements of 
$\Omega$ the \textit{ends} of $\Delta$, in analogy with the rank~$1$ case where 
$\Delta$ is a tree and $\Omega$ is the set of ends in the usual sense. 

The \textit{dominant end} $\varpi$ is the equivalence class of~$\fs_0$.
\end{defn}

\begin{defn}
For $x$ a special vertex and $\omega\in\Omega$, let $\fs_x^{\omega}$ denote the unique 
sector of $\Delta$ based at $x$ and in the class of the end~$\omega$ (see \cite{AB}). 
We say that a sequence $(x_n)_{n\geq 0}$ of points in $\Delta$ \textit{converges to the end} 
$\omega\in \Omega$ if the distance from the set $\fs_o^{\omega}\cap\fs_{x_n}^{\omega}$ 
to the boundary of the sector $\fs_o^{\omega}$ tends to infinity as $n$ tends to infinity. 
\end{defn}

In other words, the sequence $(x_n)_{n\geq 0}$ converges to the end $\omega$ if the set $\fs_o^{\omega}\cap \fs_{x_n}^{\omega}$ moves deeper and deeper 
into the sector $\fs_o^{\omega}$, and away from all the walls of this sector. It is 
immediate that this definition does not depend on the choice of base vertex~$o$.  
It can also be shown that the limit of a sequence, if it exists, is unique. That is, 
if $x_n\to\omega$ and $x_n\to\omega'$ in $\Omega$, then $\omega=\omega'$.

In the rank~$1$ case (where $\Delta$ is a tree), the combinatorial boundary and the 
visibility boundary coincide, and the respective notions of convergence agree. 
However, in higher rank the two notions of convergence are rather different. 
Suppose that $(x_n)_{n\geq 0}$ converges to the ideal point~$\xi$ in the visibility 
boundary~$\partial\Delta$. Let $\fr:[0,\infty)\to \Delta$ be the unique ray based 
at $o$ in the class~$\xi$. If the direction $\vect{d}\bigl(o,\fr(1)\bigr)$ of $\fr$ is a 
\textit{regular} vector (a vector $\lambda\in E$ is regular if 
$\langle\lambda,\alpha\rangle\neq 0$ for all $\alpha\in R$) then one can show that 
$(x_n)_{n\geq 0}$ also converges in the combinatorial boundary, and the limit point 
is the equivalence class of the unique sector based at $o$ containing the image 
of~$\fr$. On the other hand, if $\vect{d}\bigl(o,\fr(1)\bigr)$ is not regular then the 
sequence $(x_n)_{n\geq 0}$ may or may not converge in the combinatorial boundary. This 
occurs even in a thin building (consisting of a single apartment). 

\begin{lemma}\label{lem:recurrence}
Let $(X_n)_{n\geq 0}$ be an irreducible, drift-free, translation invariant random walk 
on the coweight lattice~$P$ with finite first moment. Then, with probability~$1$, 
the random walk $(X_n)_{n\geq 0}$ visits every hyperplane $H_{\alpha,m}$ 
(with $\alpha\in R$ and $m\in\mathbb{Z}$) infinitely often.
\end{lemma}

\begin{proof}
Let $\sigma(\lambda)=p(0,\lambda)$.
Let $\alpha\in R$ be fixed, and let $Y_n=\langle X_n,\alpha\rangle$. Thus $(Y_n)_{n\geq 0}$ 
is a random walk on $\mathbb{Z}$, because $\langle \lambda,\alpha\rangle\in\mathbb{Z}$ for all 
$\lambda\in P$. The law of $(Y_n)_{n\geq 0}$ is $\tilde{\sigma}(m)=\sigma(H_{\alpha,m})$. 
Note that $Y_n=m$ if and only if $X_n\in H_{\alpha,m}$. The drift of the random 
walk~$(Y_n)_{n\geq 0}$ is
\begin{align*}
\sum_{m\in\mathbb{Z}}\tilde{\sigma}(m)\,m
&=\sum_{m\in\mathbb{Z}}\sum_{\;\,\lambda\in H_{\alpha,m}\cap P}\sigma(\lambda)\,m
=\sum_{\lambda\in P}\sigma(\lambda)\, \langle\lambda,\alpha\rangle
=\left\langle\sum_{\lambda\in P}\sigma(\lambda)\,\lambda\,,\,\alpha\right\rangle=0
\end{align*}
(since $X_n$ is drift-free), and this calculation also shows that $\tilde{\sigma}$ has finite first moment. 
Thus, by the classical theory $(Y_n)_{n\geq 0}$ is a recurrent random walk. 
Hence $(X_n)_{n\geq 0}$ visits each hyperplane $H_{\alpha,m}$ infinitely often with probability~$1$.
\end{proof}

\begin{thm}\label{thm:roe0} Let $\Delta$ be a locally finite thick regular affine building 
of type other than $\tilde{E}_8$, $\tilde{F}_4$, or $\tilde{G}_2$. Let $(X_n)_{n\geq 0}$ 
be a nearest neighbour semi-isotropic random walk on $\Delta$. Suppose that $\mu=0$ 
(with $\mu$ from Theorem~\ref{thm:roe}), and that the random walk 
$(\overline{X}_n)_{n\geq 0}$ is irreducible on~$P$ (this occurs, for example, if $(X_n)_{n\geq 0}$ 
is irreducible on $V_P$). 

Then $X_n$ converges almost surely in the sense of Definition \ref{def:ends} 
to the dominant end $\varpi$ of~$\Delta$.  
\end{thm}

\begin{proof} In the proof we write points $\lambda\in\Sigma$ as vectors in terms of the basis 
$\omega_1,\ldots,\omega_r$ of fundamental coweights as $\lambda=(\lambda^1,\ldots,\lambda^r)$. 
That is, $\lambda=\lambda^1\omega_1+\cdots+\lambda^r\omega_r$, and 
$\lambda^i=\langle \lambda,\alpha_i\rangle$. 

We first set up some of the geometry needed for the proof. For special vertices $x\in V_P$, 
let $\fs_x$ be the unique sector of $\Delta$ based at $x$ in the class $\varpi=[\fs_0]$, and 
consider the intersection $\fs_x\cap \Sigma$. This set can be written as the intersection 
of half apartments, and thus is of the form
\begin{align}\label{eq:ks}
\Sigma\cap\fs_x=\bigcap_{\alpha\in R^+}H_{\alpha\,,\,k_{\alpha}}^+\qquad
\textrm{for some integers $k_{\alpha}$}.
\end{align}
Let $\pi(x)\in\Sigma$ be the point $\pi(x)=\bigl(\pi(x)^1,\ldots,\pi(x)^r\bigr)$, where
$$
\pi(x)^i=\min\{k\in\mathbb{Z}\mid  H_{\alpha_i\,,\,k}\cap\fs_x\neq\emptyset\}\qquad
\textrm{for each $1\leq i\leq r$}.
$$
Thus $\pi(x)^i=k_{\alpha_i}$, with $k_{\alpha}$ as in~(\ref{eq:ks}). These definitions 
are illustrated for type~$\tilde{A}_2$ in Figure~\ref{fig:intersection}. 

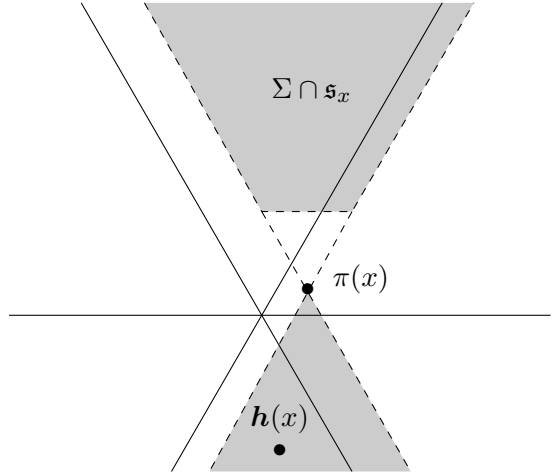
\begin{figure}[h]
\begin{center}
\begin{tikzpicture} [scale=1.2]
\path [fill=lightgray] (-1.3,3.464) -- (0,1.15) -- (1,1.15) -- (2.35,3.464) -- (-1.3,3.464);
\path [fill=lightgray] (-0.58,-1.732) -- (0.51,0.263) -- (1.65,-1.732);
\draw [thin, dashed] (-1.3,3.464) -- (0,1.15) -- (1,1.15) -- (2.35,3.464);
\draw [thin, dashed] (0,1.15) -- (1.65,-1.732);
\draw [thin, dashed] (1,1.15) -- (-0.58,-1.732);
\draw (-2,3.464) -- (1,-1.732);
\draw (-1,-1.732) -- (2,3.464);
\draw (-2.8,0) -- (3.2,0);
\node at (0.55,2.5) {$\Sigma\cap\fs_x$};
\node at (0.2,-1.5) {$\bullet$};
\node at (0.2,-1.15) {$\vect{h}(x)$};
\node at (1.1,0.4) {$\pi(x)$};
\node at (0.51,0.28) {$\bullet$};
\end{tikzpicture}
\caption{The intersection $\Sigma\cap\fs_x$ and the point $\pi(x)$}\label{fig:intersection}
\end{center}
\end{figure}

If $A$ is an apartment containing $\fs_x$ then $\vect{h}|_A:A\to \Sigma$ is the isomorphism 
fixing the intersection $A\cap \Sigma$, and since $x$ is the base point of $\fs_x\,$, we 
necessarily have that $\pi(x)-\vect{h}(x)$ is dominant. Thus $\vect{h}(x)^i\leq\pi(x)^i$ 
for all $1\leq i\leq r$ (and so $\vect{h}(x)$ lies in the lower shaded region in 
Figure~\ref{fig:intersection}).

If $x$ and $y$ are distinct special vertices lying in a common chamber we write $x\approx y$. 
We need to analyse  the connection between $\pi(x)$ and $\pi(y)$ for vertices $x\approx y$. 
It is easier to consider each component of these vectors separately, so let $1\leq i\leq r$ 
be fixed. We claim that for $x\approx y$,
\begin{enumerate}
\item If $\vect{h}(x)^i<\pi(x)^i$ then $\pi(y)^i=\pi(x)^i$.
\item If $\vect{h}(x)^i=\pi(x)^i$ and 
\begin{enumerate}
\item if $\vect{h}(y)^i=\vect{h}(x)^i+1$ then $\pi(y)^i=\pi(x)^i+1$.
\item if $\vect{h}(y)^i=\vect{h}(x)^i$ then $\pi(y)^i=\pi(x)^i$.
\item if $\vect{h}(y)^i=\vect{h}(x)^i-1$ then $\pi(y)^i\in\{\pi(x)^i-1,\pi(x)^i\}$.
\end{enumerate}
\end{enumerate}
In the last case, the proportion of the $y'\approx x$ with $\vect{h}(y')=\vect{h}(y)$ having $\pi(y')^i=\pi(x)^i-1$ is~$q_i^{-1}$.

Let us assume the above claims for now.  Let $\overline{X}_n=\vect{h}(X_n)$ and 
$Y_n=\pi(X_n)$, and write $Z_n=Y_{n+1}-Y_n$ for the increments of~$Y_n$. We are interested 
in the process $(Y_n)_{n\geq 0}$, and we show that each $Y_n^i\to\infty$ almost surely 
as~$n\to\infty$. This clearly implies that $X_n$ converges in the combinatorial boundary 
to the equivalence class of~$\fs_0$.

The process $(\overline{X}_n)_{n\geq 0}$ is a translation invariant drift free random walk 
on~$P$, and by Lemma~\ref{lem:recurrence} this random walk hits each fixed hyperplane 
$H_{\alpha,k}$ infinitely often almost surely. Thus for each $k\in\mathbb{Z}$ and each 
$1\leq i\leq r$ we have $\overline{X}_n^i=k$ for infinitely many $n\geq 0$ almost surely.

Let $1\leq i\leq r$ be fixed. We have $\overline{X}_n^i\leq Y_n^i$ for all $n\geq 0$. 
The claims above imply the following. If $\overline{X}_n^i<Y_n^i$ then $Z_n^i=0$,
while if $\overline{X}_n^i=Y_n^i$ then
\begin{equation} \label{eq:condi}
\begin{aligned}
\mathbb{P}[Z_n^i=1\mid \overline{X}_{n+1}^i=\overline{X}_n^i+1]&=1,&
\mathbb{P}[Z_n^i=0\mid \overline{X}_{n+1}^i=\overline{X}_n^i]&=1,\\
\mathbb{P}[Z_n^i=0\mid\overline{X}_{n+1}^i=\overline{X}_n^i-1]&=1-q_i^{-1},&
\mathbb{P}[Z_n^i=-1\mid\overline{X}_{n+1}^i=\overline{X}_n^i-1]&=q_i^{-1}.
\end{aligned}
\end{equation}
Thus $Y_n^i$ remains constant (in $n$) until $\overline{X}_n^i=Y_n^i$.
We can therefore define a sequence of stoping times $\bt(k)$ ($k \ge 0$) depending on $i$ by 
$$
\bt(0)=0 \quad \text{and} \quad \bt(k+1) = \inf\{ n > \bt(k) : \overline{X}_n^i=Y_n^i\}.
$$
Then $Y_n^i = Y_{\bt(k)+1}^i$ for all $k \ge 0$ and $\bt(k) < n \le \bt(k+1)$, so that
we can argue inductively: by recurrence, $\overline{X}_n^i$ must reach $Y_{\bt(k)+1}^i$
after time $\bt(k)$ with probability~1, whence $\bt(k+1)$ is almost surely finite. 
For $k \ge 1$, the random variables
$$
Z_{\bt(k)}^i = Y_{\bt(k)+1}^i - Y_{\bt(k)}^i = Y_{\bt(k+1)}^i - Y_{\bt(k)}^i 
= \overline{X}_{\bt(k+1)}^i - \overline{X}_{\bt(k)}^i 
$$ 
are independent. Using (\ref{eq:condi}), we easily compute the distribution of 
$Z_{\bt(k)}^i \in \{ -1,0,1\}$ via the law of total probability:
$$
\mathbb{P}[Z_{\bt(k)}^i = 1] = \mathbb{P}[\overline{X}_{\bt(k)+1}^i-\overline{X}_{\bt(k)}^i=1] 
\quad \text{and} \quad \mathbb{P}[Z_{\bt(k)}^i = -1] 
= q_i^{-1}\,\mathbb{P}[\overline{X}_{\bt(k)+1}^i-\overline{X}_{\bt(k)}^i=-1]\,,
$$ 
while of course $\mathbb{P}[Z_{\bt(k)}^i = 0]$ is the complementary probability.
Thus, the $Z_{\bt(k)}^i$ are i.i.d. 
Now, since $(\overline{X}_{n})_{n\geq0}$ is drift free, the projection $(\overline{X}_n^i)_{n\geq 0}$ is 
also drift free, and so $\mathbb{E}[\overline{X}_{n+1}^i-\overline{X}_n^i]=0$. These increments
are $\pm 1$ or $0$. Thus 
$\mathbb{P}[\overline{X}_{n+1}^i-\overline{X}_n^i=-1]=
\mathbb{P}[\overline{X}_{n+1}^i-\overline{X}_n^i=1]$, 
and by irreducibility this value is strictly positive. Since $q_i\geq 2$ (by thickness),
\begin{equation*}\
\mathbb{E}[Z_{\bt(k)}^i]= (1-q_i^{-1})\,\mathbb{P}[\overline{X}_{n+1}^i-\overline{X}_n^i=1] = e_i>0\,,
\end{equation*}
where $e_i>0$ is a constant not depending on~$n$. Therefore
$$
Y_n^i=\sum_{\{k:\bt(k) <n\}} Z_{\bt(k)}^i \to \infty \qquad \text{almost surely} 
$$ 
by the classical Law of Large Numbers, proving the theorem.

\smallskip

It remains to prove the geometric claims 1 and 2 made above. We will briefly sketch the 
proof of the most delicate part, case 2(c). The remaining cases are easier. 
Suppose that $x\approx y$ with $\vect{h}(x)^i=\pi(x)^i$ and $\vect{h}(y)^i=\vect{h}(x)^i-1$. 
The assumption that $\vect{h}(x)^i=\pi(x)^i$ says that the set $\Sigma\cap \fs_x$ is adjacent 
to a wall of the sector $\fs_x$, and the assumption that $\vect{h}(y)^i=\vect{h}(x)^i-1$ says 
that $y$ is on the `negative side' of this wall (in any apartment containing $\fs_x\cup\{y\}$, 
see Lemma~\ref{lem:config}). This is illustrated in Figure~\ref{fig:intersection1}.

\begin{figure}[h]
\begin{center}
\begin{tikzpicture} [scale=1.2]
\path [fill=lightgray] (-1.3,3.464) -- (0,1.15) -- (1,1.15) -- (2.35,3.464) -- (-1.3,3.464);
\path [fill=lightgray] (1.5,2) -- (2.3,2) -- (1.9,2.7);
\draw (-1.3,3.464) -- (0,1.15) -- (1,1.15) -- (2.35,3.464);
\draw (1,1.15) -- (-0.16,-1);
\draw (-2.65,3.464) -- (0.25,-1.732);
\draw (-0.16,-1) -- (0.65,-1);
\node at (0.55,2.5) {$\Sigma\cap\fs_x$};
\node at (-1.5,2.5) {$\fs_x$};
\node at (-0.16,-1) {$\bullet$};
\draw (0.65,-1) -- (0.22,-0.3);
\draw (0.65,-1) -- (0.25,-1.732);
\node at (0.25,-2) {$y$};
\draw (0.22,-0.3) -- (-0.55,-0.3);
\node at (-0.5,-1) {$x$};
\node at (0.25,-1.732) {$\bullet$};
\draw (1.5,2) -- (2.3,2) -- (1.9,2.7);
\node at (0.25,-0.75) {$c_1$};
\node at (-0.18,-0.5) {$c_0$};
\node at (0.25,-1.25) {$d$};
\node at (1.9,2.25) {$c$};
\end{tikzpicture}
\caption{Evolution of $\pi(x)$}\label{fig:intersection1}
\end{center}
\end{figure}
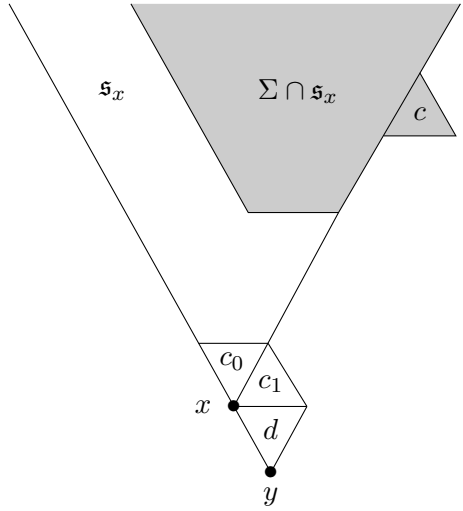

Let $d$ be a chamber containing $x$ and $y$ (this chamber exists by hypothesis $x\approx y$). 
Let $c_0$ be the base chamber of $\fs_x$, and let $c_0\sim c_1\sim\cdots\sim c_k=d$ be a minimal 
length gallery from $c_0$ to $d$. Let $c$ be a chamber of the base apartment~$\Sigma$ adjacent to 
the region $\Sigma\cap \fs_x$, as illustrated. Then $c\in \fs_y$ if and only if the chamber $c_1$ 
is equal to the projection $\mathrm{proj}_c(x)$ of the chamber $c$ onto the simplex~$x$ 
(see \cite[Proposition~4.95]{AB}). There are $q_i$ chambers adjacent to $c_0$ via the panel 
$c_0\cap c_1$, whence the result.
\end{proof}

\medskip

\noindent\begin{minipage}{0.5\textwidth}
 James Parkinson\newline
School of Mathematics and Statistics\newline
University of Sydney\newline
Carslaw Building, F07\newline
NSW, 2006, Australia\newline
\texttt{jamesp@maths.usyd.edu.au}
\end{minipage}
\begin{minipage}{0.5\textwidth}
\noindent Wolfgang Woess\newline
Institut f\"{u}r Mathematische Strukturthorie\newline
Technische Universit\"{a}t Graz\newline
Steyrergasse 30\newline
8010 Graz, Austria\newline
\texttt{woess@TUGraz.at}
\end{minipage}

\end{document}